\documentclass[11pt,a4paper]{article}
\pdfoutput=1
\title{Computing period matrices and the Abel-Jacobi map of superelliptic curves}
\author{Pascal Molin, Christian Neurohr}

\usepackage[a4paper,width=150mm,top=25mm,bottom=25mm]{geometry}
\usepackage[utf8]{inputenc}
\usepackage[english]{babel}
\usepackage[T1]{fontenc}
\usepackage{amsmath,amsfonts}
\usepackage{amsthm}
\usepackage{amssymb}
\usepackage{mathtools}
\usepackage{lmodern}
\usepackage{graphicx}
\usepackage{booktabs}
\usepackage{float}
\usepackage{subfig}
\usepackage{cite}
\usepackage{fancyhdr}
\usepackage{multirow}
\usepackage{nicefrac}
\usepackage{wrapfig}
\usepackage{color}
\usepackage{standalone}
\usepackage{bigints}
\usepackage{varioref}
\usepackage{tikz}
\usepackage{xr}
\usepackage{hyperref}

\newcommand{\eq}{\Leftrightarrow}

\newcommand{\To}{\longrightarrow}

\newcommand{\isom}{\cong}

\newcommand{\Z}{\mathbb{Z}}

\newcommand{\R}{\mathbb{R}}
\newcommand{\C}{\mathbb{C}}

\renewcommand{\P}{\mathbb{P}}

\DeclarePairedDelimiter\floor{\lfloor}{\rfloor}
\newcommand\set[1]{\left\{#1\right\}}
\newcommand\abs[1]{\left|#1\right|}

\DeclareMathOperator{\asinh}{asinh}

\DeclareMathOperator{\Jac}{Jac}
\DeclareMathOperator{\Aut}{Aut}
\DeclareMathOperator{\GL}{GL}
\DeclareMathOperator{\PSl}{PSL}
\DeclareMathOperator{\pr}{pr_x} 
\DeclareMathOperator{\ord}{ord}
\renewcommand{\Re}{\operatorname{Re}}
\renewcommand{\Im}{\operatorname{Im}}
\renewcommand{\div}{\operatorname{div}}
\newcommand{\Div}{\operatorname{Div}}
\renewcommand{\deg}{\operatorname{deg}}
\newcommand{\Prin}{\operatorname{Prin}}
\newcommand{\supp}{\operatorname{supp}}
\renewcommand{\ord}{\operatorname{ord}}

\newcommand{\cu}{\mathcal{C}} 
\newcommand{\caff}{\mathcal{C}_{\text{aff}}} 
\newcommand{\cafft}{\tilde{\mathcal{C}}_{\text{aff}}} 
\newcommand{\X}{\{  x_1,\dots,x_d \}} 
\renewcommand\d{\mathrm{d}\,}
\newcommand{\dx}{\mathrm d x}

\newcommand{\dt}{\mathrm d t}
\newcommand{\du}{\mathrm d u}
\newcommand{\w}{\omega} 
\newcommand{\wtij}{\omega_{\tilde{i},j}}
\newcommand{\W}{\mathcal{W}} 
\newcommand{\WM}{\mathcal{W}^{\text{mer}}} 
\newcommand{\hd}{\Omega^1_{\mathcal{C}}} 
\newcommand{\homo}{H_1(\cu,\Z)} 
\newcommand{\cyab}{\gamma_{a,b}}
\newcommand{\cycd}{\gamma_{c,d}}
\newcommand{\cyad}{\gamma_{a,d}}
\newcommand{\cybd}{\gamma_{b,d}}
\newcommand{\cyabl}{\gamma_{a,b}^{(l)}}
\newcommand{\cyabk}{\gamma_{a,b}^{(k)}}

\newcommand{\cycdl}{\gamma_{c,d}^{(l)}}
\newcommand{\cyadl}{\gamma_{a,d}^{(l)}}

\newcommand{\OA}{\Omega_A} 
\newcommand{\OB}{\Omega_B} 
\newcommand{\OC}{\Omega_{\Gamma}} 
\newcommand{\AJ}{\mathcal{A}} 
\newcommand{\cab}{C_{a,b}}
\newcommand{\ccd}{C_{c,d}}
\newcommand{\cbd}{C_{b,d}}
\newcommand{\cad}{C_{a,d}}
\newcommand{\yab}{y_{a,b}}
\newcommand{\ycd}{y_{c,d}}

\newcommand{\yaxp}{y_{a,x_P}}
\newcommand{\mr}{^{\frac{1}{m}}}
\newcommand{\vab}{V_{a,b}}

\newcommand{\vcd}{V_{c,d}}
\newcommand{\uab}{u_{a,b}}
\newcommand{\ubd}{u_{b,d}}
\newcommand{\uad}{u_{a,d}}
\newcommand{\ucd}{u_{c,d}}
\newcommand{\uaxp}{u_{a,x_P}}
\newcommand{\xab}{x_{a,b}}
\newcommand{\xad}{x_{a,d}}
\newcommand{\xbd}{x_{b,d}}
\newcommand{\xaxp}{x_{a,x_P}}
\newcommand{\yt}{\tilde{y}}
\newcommand{\ytab}{\tilde{y}_{a,b}}
\newcommand{\ytcd}{\tilde{y}_{c,d}}
\newcommand{\ytbd}{\tilde{y}_{b,d}}
\newcommand{\ytad}{\tilde{y}_{a,d}}
\newcommand{\ytaxp}{\tilde{y}_{a,x_P}}
\newcommand{\xt}{\tilde{x}}
\DeclareMathOperator{\dist}{dist}

\newcommand\cmul{\mathcal M}
\newcommand\ctrig{\mathcal T}
\newcommand\ctot{\mathcal E}

\renewcommand\d{\mathrm{d}}

\DeclareUnicodeCharacter{3B1}{\alpha}
\DeclareUnicodeCharacter{3BB}{\lambda}
\DeclareUnicodeCharacter{3C4}{\tau}
\DeclareUnicodeCharacter{3C0}{\pi}
\DeclareUnicodeCharacter{3B7}{\eta}
\DeclareUnicodeCharacter{3C6}{\varphi}
\DeclareUnicodeCharacter{3B1}{\alpha}
\DeclareUnicodeCharacter{3B2}{\beta}
\DeclareUnicodeCharacter{3BB}{\lambda}
\DeclareUnicodeCharacter{3B3}{\gamma}
\DeclareUnicodeCharacter{3B4}{\delta}
\DeclareUnicodeCharacter{3C4}{\tau}
\DeclareUnicodeCharacter{3C0}{\pi}
\DeclareUnicodeCharacter{3B5}{\varepsilon}
\DeclareUnicodeCharacter{3C9}{\omega}
\DeclareUnicodeCharacter{3C1}{\rho}

\theoremstyle{plain}
\newtheorem{thm}{Theorem}[section]
\newtheorem{prop}[thm]{Proposition}
\newtheorem{lemma}[thm]{Lemma}
\newtheorem{coro}[thm]{Corollary}

\theoremstyle{definition}
\newtheorem{defn}[thm]{Definition}

\theoremstyle{remark}
\newtheorem{rmk}[thm]{Remark}

\hypersetup{linkbordercolor  ={1 1 1}}
\hypersetup{citebordercolor  ={1 1 1}}
\hypersetup{urlbordercolor  ={1 1 1}}

\newcommand{\abstand}{{\\[0.25\baselineskip]}}

\usetikzlibrary{arrows}
\usetikzlibrary{shapes.misc}
\usetikzlibrary{decorations.markings}
\tikzset{cross/.style={cross out, draw=black, minimum size=2*(#1-\pgflinewidth), inner sep=0pt, outer sep=0pt},cross/.default={1pt}}
\tikzset{
    halfarrow1/.style={postaction={decorate},
        decoration={markings,mark=at position .5 with
        {\arrow[line width=0.4mm]{>}}}} }
\tikzset{
    halfarrow2/.style={postaction={decorate},
        decoration={markings,mark=at position .5 with
        {\arrow[line width=0.4mm]{<}}}} }

\DeclareFontFamily{OMX}{lmex}{}
\DeclareFontShape{OMX}{lmex}{m}{n}{<-> lmex10}{}
\def\biblio{\bibliographystyle{plain}\bibliography{main}}

\begin{document}

\def\biblio{}

\maketitle


  \section{Introduction}

  The Abel-Jacobi map links a complex curve to a complex torus.
  In particular the matrix of periods allows to define the Riemann
  theta function of the curve, which is an object of central interest in
  mathematics and physics: let us
  mention the theory of abelian functions or integration of partial differential
  equations.

  In the context of cryptography and number theory, periods also appear
  in the BSD conjecture or as a tool to identify isogenies or to find
  curves having prescribed complex multiplication \cite{vanWamelen06}.
  For such diophantine applications, it is necessary to compute
  integrals to large precision (say thousand digits) and to have
  rigorous results.

  \subsection{Existing algorithms and implementations}

  For genus 1 and 2, methods based on isogenies (AGM \cite{CremonaAGM13},
  Richelot \cite{BostMestre88}, Borchardt mean \cite{Labrande16})
  make it possible to compute periods to arbitrary precision in almost
  linear time. However, these techniques scale very badly when the genus grows.

  For modular curves, the modular symbols machinery and termwise integration of
  expansions of modular forms give excellent algorithms
  \cite[\S 3.2]{Mascot13}.

  For hyperelliptic curves of arbitrary genus, the Magma implementation
  due to van Wamelen \cite{vanWamelen06} computes period matrices and the Abel-Jacobi map.
  However, it is limited in terms of precision (less
  than $2000$ digits) and some bugs are experienced on
  certain configurations of branch points. The shortcomings of this implementation motivated our
  work. Using a different strategy
  (integration along a tree instead of around Voronoi cells)
  we obtain a much faster, more reliable algorithm and rigorous results.

  For general algebraic curves, there is an implementation in Maple
  due to Deconinck and van Hoeij \cite{DeconinckvanHoeij01}.
  We found that this package is not suitable for high precision purposes.

  We also mention the Matlab implementations due to Frauendiener and Klein for hyperelliptic curves \cite{FrauendienerKlein2015}
    and for general algebraic curves \cite{FrauendienerKlein2011}.
  
  Moreover, a Sage implementation for general algebraic curves due to Nils Bruin and Alexandre Zotine is in progress.
  
  \subsection{Main result}

  This paper adresses the problem of computing period matrices and the
  Abel-Jacobi map of algebraic curves given by an affine equation of the form  (see Definition \ref{m-def:se_curve})
  \begin{equation*}
  y^m = f(x), \quad m > 1, f \in \C[x] \text{ separable of degree} \deg(f) = n \ge 3.
  \end{equation*}
  They generalize
  hyperelliptic curves and are usually called \textit{superelliptic curves}.

  We take advantage of their specific geometry to obtain the following
  (see Theorem \ref{m-thm:complexity_integrals})
  \begin{thm}
      Let $\cu$ be a superelliptic curve of genus $g$ defined by an equation $y^m=f(x)$
      where $f$ has degree $n$.
      We can compute a basis of the period lattice to
      precision $D$ using $$O(n(g+\log D)(g+D)^2\log^{2+\varepsilon} (g+D)) \text{ binary operations,}$$
      where $\epsilon>0$ is chosen so that
      the multiplication of precision $D$ numbers has complexity
      $O(D\log^{1+\epsilon}D)$.
  \end{thm}

  \subsection{Rigorous implementation}
  \label{subsec:arb}

  The algorithm has been implemented in C using the Arb library \cite{Johansson2013arb}.
  This system represents a complex numbers as a floating point approximation
  plus an error bound, and automatically
  takes into account all precision loss occurring through the
  execution of the program. With this model we can certify
  the accuracy of the numerical results of our algorithm (up to human or even
  compiler errors, as usual).

  Another implementation has been done in Magma \cite{Magma}. Both are publicly available
  on github at \url{https://github.com/pascalmolin/hcperiods} \cite{githubhcperiods_2017_833727}.

  \subsection{Interface with the LMFDB}

  Having rigorous period matrices is a valuable input for the methods developed by
  Sijsling et al. \cite{CMSVEndos} to compute endormorphism rings of Jacobians of hyperelliptic
  curves.
  During a meeting aimed at expanding the `L-functions and modular forms database' \cite[LMFDB]{lmfdb}
  to include genus $3$ curves, the Magma implementation of our algorithm was incorporated in their framework
  to successfully compute the endomorphism rings of Jacobians of $67,879$ hyperelliptic
  curves of genus $3$, and confirm those of the $66,158$ genus
  2 curves that are currently in the database.

  For these applications big period matrices were computed to $300$ digits precision.

  \subsection{Structure of the paper}

  In Section \ref{m-sec:ajm} we briefly review the objects we are interested
  in, namely period matrices and the Abel-Jacobi map of nice algebraic curves.
  The ingredients to obtain these objects, a basis of holomorphic differentials
  and a homology basis, are made explicit in the case of superelliptic curves
  in Section \ref{m-sec:se_curves}.
  We give formulas for the computation of periods in Section
  \ref{m-sec:strat_pm} and explain how to obtain from them the standard period
  matrices using symplectic reduction.
  In Section \ref{sec:intersections} we give explicit formulas for the
  intersection numbers of our homology basis.
  For numerical integration we employ two different integration schemes that
  are explained in Section \ref{m-sec:numerical_integration}: the
  double-exponential integration and
  (in the case of hyperelliptic curves) Gauss-Chebychev integration.
  The actual computation of the Abel-Jacobi map is explained in detail in
  Section \ref{m-sec:comp_ajm}.
  In Section \ref{m-sec:comp_asp} we analyze the complexity of our algorithm
  and share some insights on the implementation.
  Section \ref{m-sec:examples_timings} contains some tables with
  running times to demonstrate the performance of the code.
  Finally, in Section \ref{m-sec:outlook} we conclude with an outlook on what can be done in the future.
  
  \subsection{Acknowledgements}

  The first author wants to thank the crypto team at Inria Nancy, where
  a first version of this work was carried out in the case of hyperelliptic
  curves. He also acknowledges the support from Partenariat Hubert Curien under
  grant 35487PL.

  The second author wants to thank Steffen Müller and Florian Hess for helpful discussions.
  Moreover, he acknowledges the support from DAAD under grant 57212102.

  \section{The Abel-Jacobi map}\label{sec:ajm}

  We recall, without proof, the main objects we are interested in, and which
  will become completely explicit in the case of superelliptic curves.
  The exposition follows that of \cite[Section 2]{vanWam1998}.

  \subsection{Definition}

  Let $\cu$ be a smooth irreducible projective curve of genus $g>0$. Its space
  of holomorphic differentials $\hd$ has dimension $g$; let us fix
  a basis $ω_1,\dots ω_g$ and denote by $\bar\w$ the vector
  $(ω_1,\dots ω_g)$.

  For any two points $P,Q \in \cu$ we can
  consider the vector integral $\int_{P}^Q\bar\w\in\C^g$, whose value
  depends on the chosen path from $P$ to $Q$.

  In fact, the integral depends on the path up to homology,
  so we introduce the {\em period lattice} of $\cu$
  \begin{equation*}
      \Lambda = \set{\int_γ ω_j, γ\in\homo} \subset\C^g,
  \end{equation*}
  where $\homo \isom\Z^{2g}$ is the first homology group
  of the curve.

  Now the integral
  \begin{equation*}
      P,Q \mapsto \int_{P}^Q \bar\w \in \C^g/\Lambda
  \end{equation*}
  is well defined, and the definition can be extended
  by linearity to the group of
  degree zero divisors
  \begin{equation*}
      \Div^0(\cu)=\set{ \sum a_i P_i, a_i\in\Z, \sum a_i = 0}.
  \end{equation*}

  The Abel-Jacobi theorem states that one obtains a
  surjective map 
  whose kernel
  is formed by divisors of functions, so that the integration
  provides an explicit isomorphism
  \begin{equation*}\label{eq:ajm_def}
      \AJ:\left\{\begin{array}{ccc}
              \Jac(\cu) = \Div^0(\cu)/\Prin^0(\cu) &\To &\C^g/\Lambda \\
              \sum_i [Q_i-P_i] &\mapsto & \sum_k \int_{P_i}^{Q_i} \bar\w \mod \Lambda
  \end{array}\right.
  \end{equation*}
  between the Jacobian variety and the complex torus. \\

  \subsection{Explicit basis and standard matrices}\label{subsec:bases_matrices}

  Let us choose a symplectic basis of $\homo$, that is two
  families of cycles $α_i$, $β_j$ for $1\leq i,j\leq g$ such that
  the intersections satisfy
  \begin{equation*}
      \left( \alpha_i \circ \beta_j \right) = \delta_{i,j},
  \end{equation*}
  the other intersections all being zero.

  We define the period matrices on those cycles
  \begin{equation*}
      \OA = \left(\int_{α_i}ω_j\right)_{1\leq i,j\leq g}
      \text{ and }
      \OB = \left(\int_{β_i}ω_j\right)_{1\leq i,j\leq g}
  \end{equation*}
  and call the concatenated matrix \\
  \begin{equation*}
      \Omega = (\OA, \OB) \in \C^{g\times 2g}
  \end{equation*}
  such that $\Lambda = \Omega\Z^{2g}$ a {\em big period matrix}. 

  If one takes as basis of differentials the dual basis of
  the cycles $α_i$, the matrix becomes
  \begin{equation*}
      \OA^{-1}\Omega = (I_g, \tau),
  \end{equation*}
  where $\tau = \OA^{-1}\OB \in \C^{g \times g}$, called  a {\em small period matrix}, is in the Siegel space
  $\mathcal{H}_g$ of symmetric matrices with positive definite imaginary part.

  \section{Superelliptic curves}\label{sec:se_curves}

  \subsection{Definition \& properties}\label{subsec:se_def}

    \begin{defn}\label{def:se_curve}
    In this paper, a superelliptic curve $\cu$ over $\C$ is a smooth projective curve that has an affine model given by an equation of the form
   \begin{equation}\label{eq:aff_model}
    \caff : \quad y^m = f(x) =  c_f \cdot \prod_{k=1}^n (x-x_k),
   \end{equation}
   \end{defn}
   where $m > 1$ and $f \in \C[x]$ is separable of degree $n \ge 3$. Note that we do not assume that $\gcd(m,n) = 1$.

  There are $\delta = \gcd(m,n)$ points $P_{\infty}^{(1)},\dots,P_{\infty}^{(\delta)} \in \cu$ at infinity, that behave differently depending on $m$ and $n$ (see \cite[\S 1]{CT1996} for details).
  In particular, $\infty \in \P^1_{\C}$ is a branch point for $\delta \ne m$. Thus, we introduce the set of finite branch points $X = \X$ as well as the set of all branch points
  \begin{align}\label{eq:branch_points}
         \hat{X} = \begin{cases}   X \cup \{ \infty \} \quad \text{if} \; m  \nmid  d,\\
         X \hfill \text{otherwise.}
     \end{cases}
  \end{align}
  The ramification indices at the branch points are given by $e_x = m$ for all $x \in X$ and $e_{\infty} = \frac{m}{\delta}$. Using the
  Riemann-Hurwitz formula, we obtain the genus of $\cu$ as
  \begin{equation}\label{eq:genus}
    g = \frac{1}{2}( (m-1)(n-1) - \delta + 1).
  \end{equation}
  We denote the corresponding finite ramification points $P_k = (x_k,0) \in \cu$ for $k = 1,\dots,n$.

  \begin{rmk}
   Without loss of generality we may assume $c_f = 1$ (if not, apply the transformation $(x,y) \mapsto (x,\sqrt[m]{c_f}y)$).
  \end{rmk}
  \begin{rmk}
      \label{rmk:moebius}
      For any $\begin{pmatrix}a&b\\c&d\end{pmatrix}\in\PSl(2,\C)$,
      the Moebius transform $\phi:u\mapsto \frac{au+b}{cu+d}$ is an automorphism
      of $\P^1$. By a change of coordinate $x=\phi(u)$ we obtain a different model of $\cu$
      given by the equation
      \begin{equation*}
          \tilde v^m = \tilde f(u)
      \end{equation*}
      where $\tilde f(u)=f(\phi(u))(cu+d)^{\ell m}$ and $v=y(cu+d)^\ell$ for
      the smallest value $\ell$ such that $\ell m\geq n$.

      If the curve was singular at infinity, the singularity is moved to $u=-d/c$ in the new model.
      This happens when $\delta < m$ (so that $\ell m > n$).

  When $\delta=m$ we may apply such a transformation to improve the configuration
  of affine branch points.
  \end{rmk}

  \subsection{Complex roots and branches of the curve}\label{subsec:roots_branches}

  \subsubsection{The complex $m$-th root}

  Working over the complex numbers we encounter several multi-valued functions
  which we will briefly discuss here.  Closely related to superelliptic
  curves over $\C$ is the complex $m$-th root.
  Before specifying a branch it is a multi-valued function $y^m = x$
  that defines an $m$-sheeted Riemann surface, whose only branch points
  are at $x = 0,\infty$, and these are totally ramified.

  For $x\in\C$, it is natural and computationally convenient to use the
  \emph{principal branch} of the $m$-th root $\sqrt[m]x$ defined by
  \begin{equation*}
      -\frac{π}m<\arg(\sqrt[m]x)\leq\frac{π}m
  \end{equation*}
  which has a branch cut along the negative real axis $]\!-\infty,0]$.
  Crossing it in positive orientation corresponds to multiplication by
  the primitive $m$-th root of unity
  \begin{equation*}
  \zeta := \zeta_m := e^{\frac{2\pi i }{m}}
  \end{equation*}
  on the surface. In
  particular, the monodromy at $x=0$ is cyclic of order $m$.

  \subsubsection{The Riemann surface}\label{subsec:riemann_surface}

  For an introduction to the theory of Riemann surfaces, algebraic curves and holomorphic covering maps we recommend \cite{Miranda1995}.

   Over $\C$ we can identify the curve $\cu$ with the compact Riemann surface $\cu(\C)$. Since our defining equation
   has the nice form
   $y^m = \prod_{k = 1}^n (x - x_k)$ we view $\cu$ as a Riemann surface with $m$ sheets and
   all computations will be done in the $x$-plane.

   We denote by $\pr : \cu \rightarrow \P^1_{\C}$ the corresponding smooth cyclic branched covering of the projective line
   defined by the $x$-coordinate.

  There
  are $m$ possibilities to continue $y$ as an analytic function following a path in the $x$-plane. This is crucial for the integration
  of differentials on $\cu$. Due to the cyclic structure of $\cu$, they are related in a convenient way:

  We call a \emph{branch of $\cu$} a function $y(x)$ such that
  $y(x)^m = f(x)$ for all $x \in \C$. At every $x$, the branches of $\cu$ only
  differ by a factor $\zeta^l$ for some $l \in \{0,\dots,m-1\}$. Thus, following a path, it is sufficient to know \emph{one} branch that is analytic in a suitable neighborhood. In the
  next paragraph, we will introduce locally analytic branches very explicitly.

  \bigskip

  Similar to the complex $m$-th root, we can assume that crossing the branch cut at $x_k \in X$ in positive direction corresponds to multiplication by $\zeta$ on the Riemann surface. We
  obtain an ordering of the sheets relative to the analytic branches of $\cu$ by imposing that multiplication by $\zeta$, i.e. applying the map
  $(x,y(x)) \mapsto (x,\zeta y(x))$, corresponds to moving one sheet up on the Riemann surface.

   \bigskip

  Consequently, the local monodromy of the cyclic covering $\pr$  is equal and cyclic of order $m$ at every $x_k \in X$
 and the monodromy group is, up to conjugation, the cyclic group $C_m$. This makes it possible to find explicit generators for the
  homology group $\homo$ without specifying a base point, as shown in \S \ref{m-subsec:cycles_homo}.

  \subsubsection{Locally analytic branches}\label{subsubsec:analytic_branches}

  In order to integrate differential forms on $\cu$
  it is sufficient to be able to follow \emph{one} explicit analytic continuation of $y$ along a
  path joining two branch points $a, b \in X$.

  One could of course consider the \emph{principal branch} of the curve
  \begin{equation*}
      y(x) = \sqrt[m]{f(x)},
  \end{equation*}
  but this is not a good model to compute with: it has branch cuts
  wandering around the $x$-plane (see Figure \ref{fig:mth_root_principal}).

  A better option is to split the product as follows:
  assume that $(a,b) = (-1,1)$. Then the function
  \begin{equation*}
      y(x) = \prod_{x_k\in X}\sqrt[m]{x-x_k}
  \end{equation*}
  has $n$ branch cuts parallel to the real line (see Figure \ref{fig:mth_root_product}).
  However, one of them lies exactly on the interval $[-1,1]$ we are interested in. We work around this
  by taking the branch cut towards $+\infty$ for each branch point $x_k$ with positive real part, writing
  \begin{equation*}
      y(x) = e^{\frac{iπr^+}m}\prod_{\Re(x_k)\leq0}\sqrt[m]{x-x_k} \prod_{\Re(x_k)>0}\sqrt[m]{x_k-x},
  \end{equation*}
  where $r^+$ is the number of points with positive real part.

  \begin{figure}[H]
      \begin{center}
          \subfloat[principal branch]{
          \includegraphics[width=.3\linewidth,page=1]{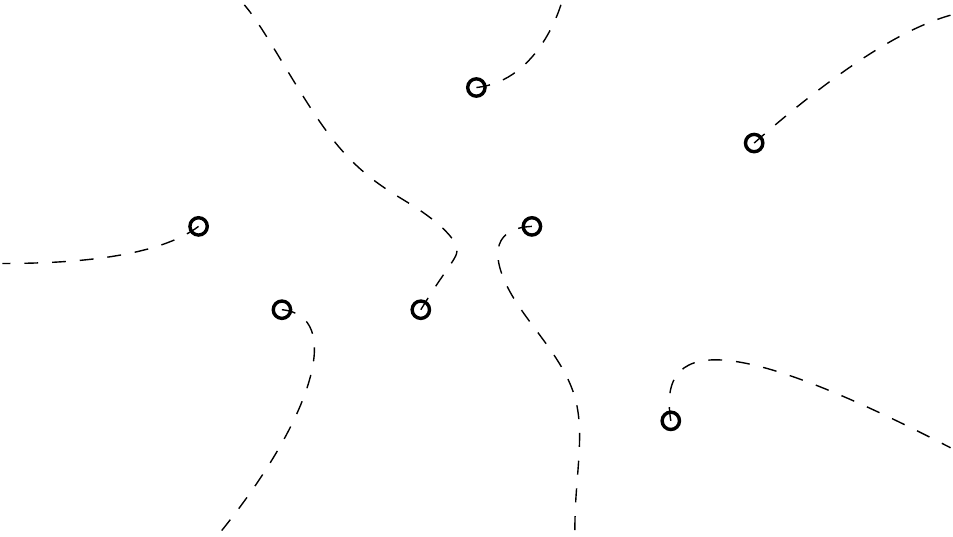}
          \label{fig:mth_root_principal} }
          \subfloat[product]{ \includegraphics[width=.3\linewidth,page=2]{branch_cuts.pdf}\label{fig:mth_root_product} }
          \subfloat[$\yab$]{ \includegraphics[width=.3\linewidth,page=3]{branch_cuts.pdf}\label{fig:mth_root_signed} }
      \end{center}
      \caption{Branch cuts of different $m$-th roots.}
  \label{fig:mth_root_pol} \end{figure}

  In general we proceed in the same way: For branch points $a,b \in X$ we consider the affine linear
  transformation
  \begin{equation*}
      \label{def:xab}
      \xab : u \mapsto \frac{b-a}{2}\left(u+\frac{b+a}{b-a}\right),
  \end{equation*}
  which maps $[-1,1]$ to the complex line segment $[a,b]$, and denote the inverse map by
  \begin{equation*}
      \label{def:uab}
      \uab : x \mapsto \frac{2x-a-b}{b-a}.
  \end{equation*}

   We split the image of the branch points under $\uab$ into the following subsets
  \begin{equation}\label{eq:uab_image}
      \set{ \uab(x), x\in X} = \set{-1,1} \cup U^+ \cup U^-,
  \end{equation}
  where points in $U^+$ (resp. $U^-$) have strictly positive (resp. non-positive) real part.

  Then the product
  \begin{equation}
      \label{eq:ytab}
      \ytab(u) = \prod_{u_k\in U^-} \sqrt[m]{u-u_k}\prod_{u_k\in U^+}\sqrt[m]{u_k-u}
  \end{equation}
  is holomorphic on a neighborhood $ε_{a,b}$ of $[-1,1]$ which we can take as
  an ellipse \footnote{we will exhibit such a neighborhood in Section \ref{m-subsec:gauss_chebychev_integration}}
  containing no point $u_k\in U^-\cup U^+$, while the term corresponding to $a,b$
  \begin{equation*}
      \sqrt[m]{1-u^2}
  \end{equation*}
  has two branch cuts $]-\infty,-1]$ and $[1,\infty[$, and is holomorphic on the complement
  $\overline U$ of these cuts.

  We can now define a branch of the curve
  \begin{equation}
      \label{eq:def_yab}
      \yab(x) =   C_{a,b} \yt_{a,b}( \uab(x) ) \sqrt[m]{1 - \uab(x)^2}
  \end{equation}
  by setting $r = 1+\#U^+ \bmod 2$ and choosing the constant
  \begin{equation}
      C_{a,b} = \left(\frac{b-a}{2}\right)^{\frac{n}{m}} e^{\frac{\pi i}{m}r}
  \end{equation}
  such that $\yab(x)^m = f(x)$.

  The function $\yab(x)$ has $n$ branch cuts all parallel to $[a,b]$ in outward direction and
  is holomorphic inside $]a,b[$ (see Figure \ref{fig:mth_root_signed}).

  More precisely, $\vab = \uab(ε_{a,b}\cap \overline U_{a,b})$,
  is an ellipse-shaped neighborhood of $]a,b[$ with two segments removed
  (see Figure \ref{m-fig:set_vab})
  on which the local branch $\yab$ is well defined and holomorphic.

  \begin{figure}[H] \begin{center}
\begin{tikzpicture}[scale=0.8, every node/.style={scale=0.8}]
  \draw (-1.5,0) node {$a$};  \draw (1.5,0) node {$b$}; \draw (0,0) node {$V_{a,b}$};
  \draw (-1.8,0) circle (1.3pt); \draw (1.8,0) circle (1.3pt);
    \draw [densely dashed] (-3,0) -- (-1.8,0);   \draw [densely dashed] (3,0) -- (1.8,0);
    \draw (-2.4,0.1) .. controls (-2,1.5) and (2,1.5) .. (2.4,0.1);
    \draw (-2.4,-0.1) .. controls (-2,-1.5) and (2,-1.5) .. (2.4,-0.1);
    \draw (-2.4,-0.1) -- (-1.7,-0.1); \draw (2.4,-0.1) -- (1.7,-0.1);
    \draw (-2.4,0.1) -- (-1.7,0.1); \draw (2.4,0.1) -- (1.7,0.1);
    \draw (-1.7,0.1) -- (-1.7,-0.1); \draw (1.7,0.1) -- (1.7,-0.1);
    \draw (-1,1.4) circle (1.3pt);
    \draw [densely dashed] (-1.1,1.4) -- (-3,1.4);
    \draw (-0.3,-1.3) circle (1.3pt);
    \draw [densely dashed] (-0.3,-1.3) -- (-3,-1.3);
    \draw (1.35,1) circle (1.3pt);
    \draw [densely dashed] (1.35,1) -- (3,1);
    \draw (0.6,-1.2) circle (1.3pt);
    \draw [densely dashed] (0.6,-1.2) -- (3,-1.2);
     \draw (0.1,1.5) circle (1.3pt);
    \draw [densely dashed] (0.1,1.5) -- (3,1.5);
\end{tikzpicture}
  \end{center} \caption{Holomorphic neighborhood of $\yab$.}
  \label{fig:set_vab} \end{figure}
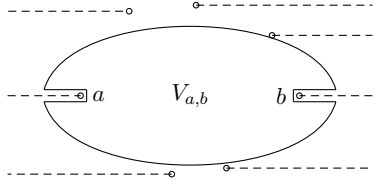

We sum up the properties of these local branches:

 \begin{prop}\label{prop:yab}
      Let $a,b\in X$ be branch points such that $X\cap\,]a,b[\,=\varnothing$.
      Then, with the notation as above, the functions
      $\ytab$ \eqref{eq:ytab} and $\yab$ \eqref{eq:def_yab}
      satisfy
     \begin{itemize}
         \item $\ytab$ is holomorphic and does not vanish on $ε_{a,b}$,
         \item $\yab(x) = C_{a,b} \ytab(\uab(x)) \sqrt[m]{1-\uab(x)^2}$ is holomorphic
         on $\vab$,
         \item $\yab(x)^m = f(x)$ for all $x\in\C$,
         \item $\yab(x),\zeta\yab(x),\dots,\zeta^{m-1}\yab(x)$ are the $m$ different analytic continuations of $y$ on $\vab$.
     \end{itemize}
     Moreover, we can assume that for $x \in \vab$, applying the map $(x,\yab(x)) \mapsto (x,\zeta^l\yab(x))$ corresponds to moving up $l \in \Z/m\Z$ sheets on the Riemann surface.
 \end{prop}

 \subsection{Cycles and homology}\label{subsec:cycles_homo}

   For us, a \emph{cycle} on $\cu$ is a smooth oriented closed path in $\pi_1(\cu)$.
   For simplicity we identify all cycles with their homology classes in $\homo = \pi_1(\cu)/[\pi_1(\cu),\pi_1(\cu)]$.

   In the following we present an
   explicit generating set of $\homo$ that relies on the locally analytic branches $\yab$ as defined in \eqref{m-eq:def_yab} and the superelliptic structure of $\cu$.

   Let $a, b \in X$ be branch points such that $X\cap]a,b[=\varnothing$, where  $[a,b]$ is the oriented line segment connecting $a$ and $b$.

   By Proposition \ref{m-prop:yab} the lifts of $[a,b]$ to $\cu$ are given by
   \begin{equation*}\label{eq:def_path_ab}
      \gamma^{(l)}_{[a,b]} = \{  (x,\zeta^l \yab(x))  \mid  x \in [a,b]  \}, \quad l \in \Z/m\Z.
   \end{equation*}
   These are smooth oriented paths that connect $P_a = (a,0)$ and $P_b = (b,0)$ on $\cu$. We obtain cycles by concatenating these lifts in the following way:
    \begin{equation}\label{eq:def_cyabl}
      \cyabl = \gamma^{(l)}_{[a,b]} \cup \gamma^{(l+1)}_{[b,a]} \in \pi_1(\cu).
   \end{equation}
   \begin{defn}[Elementary cycles]\label{def:elem_cycles}
       We say $\cyab = \cyab^{(0)}$ is an \emph{elementary cycle} and call $\cyabl$ its \emph{shifts} for $l \in \Z/m\Z$.
   \end{defn}
   In $\pi_1(\cu)$
   shifts of elementary cycles are homotopic to cycles that encircle  $a$ in negative and $b$ in positive orientation, once each.
   By definition of $\yab$ the branch cuts at the end points are outward and parallel to $[a,b]$. Thus, we have the following useful visualizations of $\cyabl$ on $\cu$:
   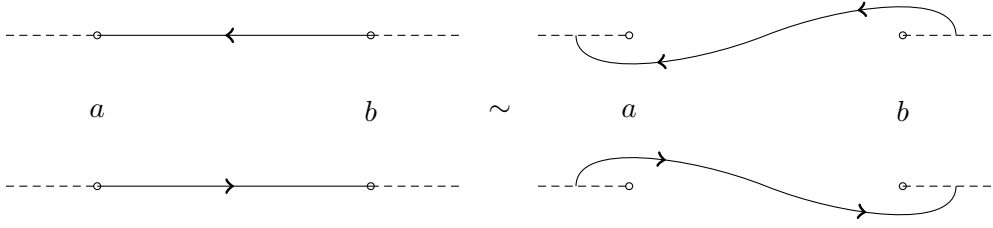
\begin{figure}[H]
      \begin{center}
\begin{tikzpicture}
     \draw (-1.8,2) node {$a$};  \draw (1.8,2) node {$b$};
     \draw [densely dashed] (-3.0,3) -- (-1.8,3);  \draw [densely dashed] (1.8,3) -- (3.0,3); \draw (-1.8,3) circle (1.3pt); \draw (1.8,3) circle (1.3pt);
      \draw (-1.8,3) -- (1.8,3) [halfarrow2];
   \draw [densely dashed] (-3.0,1) -- (-1.8,1);  \draw [densely dashed] (1.8,1) -- (3.0,1); \draw (-1.8,1) circle (1.3pt); \draw (1.8,1) circle (1.3pt);
\draw (-1.8,1) -- (1.8,1) [halfarrow1];

      \draw (3.5,2) node {$\sim$};
      
           \draw (5.2,2) node {$a$};  \draw (8.8,2) node {$b$};
     \draw [densely dashed] (4.0,3) -- (5.2,3);  \draw [densely dashed] (8.8,3) -- (10,3); \draw (5.2,3) circle (1.3pt); \draw (8.8,3) circle (1.3pt);
     
\draw (4.5,3) .. controls (4.5,2.4) and (6,2.6) .. (7,3) [halfarrow2];
\draw (7,3) .. controls (8,3.4) and (9.5,3.6) .. (9.5,3) [halfarrow2];

   \draw [densely dashed] (4,1) -- (5.2,1);  \draw [densely dashed] (8.8,1) -- (10,1); \draw (5.2,1) circle (1.3pt); \draw (8.8,1) circle (1.3pt);
\draw (4.5,1) .. controls (4.5,1.6) and (6,1.4) .. (7,1) [halfarrow1];
\draw (7,1) .. controls (8,0.6) and (9.5,0.4) .. (9.5,1) [halfarrow1];
\end{tikzpicture}
      \end{center}
    \caption{Representations of a cycle $\cyabl$.}
    \label{fig:elem_cycle}
\end{figure}

  \bigskip

  As it turns out, we do not need all elementary cycles and their shifts to
  generate $\homo$, but only those that correspond to edges in a \emph{spanning tree},
  that is a subset $E\in X\times X$ of directed edges $(a,b)$ such that all branch points
  are connected without producing any cycle. It must contain exactly $n-1$ edges.
  The actual tree will be chosen in \S \ref{subsec:spanning_tree} in order to minimize
  the complexity of numerical integration.

    \medskip
  For an edge $e = (a,b) \in E$, we denote by $\gamma_e^{(l)}$ the shifts of
  the corresponding elementary cycle $\cyab$.

  \begin{thm}\label{thm:gen_set}
      Let $E$ be a spanning tree for the branch points $X$.
   The set of cycles $\Gamma = \left\{  \gamma_{e}^{(l)}  \mid  0 \le l <m-1,  e \in E  \right\}$ generates $\homo$.
  \end{thm}
  \begin{proof}
  Denote by $\alpha_a \in \pi_1(\P^1 \setminus \hat{X})$ a closed path that encircles the branch point $a \in \hat{X}$ exactly once. Then,  due to the relation $1 = \prod_{a \in \hat{X}} \alpha_a$,
  $\pi_1(\P^1 \setminus \hat{X})$ is freely generated by $\{ \alpha_a \}_{a \in X}$, i.e. in the case $\delta \ne m$ we can omit $\alpha_{\infty}$. \abstand
  Since our covering is cyclic, we have that $
  \pi_1(\cu \setminus \pr^{-1}(\hat{X})) \isom \ker(\pi_1(\P^1 \setminus \hat{X}) \overset{\Phi}{\To} \Aut(\cu \setminus \pr^{-1}(\hat{X})))$ where $\Aut(\cu \setminus \pr^{-1}(\hat{X})) \isom C_m
  \subset S_m$
  and $\Phi(\alpha_a)$ is cyclic of order $m$ for all $a \in X$. Hence, for every word $\alpha = \alpha_1^{s_1}\dots \alpha_n^{s_n} \in \pi_1(\P^1 \setminus \hat{X})$ we have that
  $\alpha \in \ker(\Phi) \eq \sum_{i=1}^n s_i \equiv 0 \bmod m$. \abstand
  We now claim that $\pi_1(\cu \setminus \pr^{-1}(\hat{X})) = \langle  \alpha_a^{-s} \alpha_b^{s},  \alpha_a^m   \mid  s \in \Z, a,b \in X  \rangle$
  and prove this by induction on $n$: for $\alpha = \alpha_1^{s_1}$, $m$ divides $s_1$ and therefore $\alpha$ is generated by $\alpha_1^m$. For $n > 1$ we write
  $\alpha = \alpha_1^{s_1}\dots \alpha_n^{s_n} = (\alpha_1^{s_1} \dots \alpha_{n-1}^{s_{n-1}+s_n})(\alpha_{n-1}^{-s_n}\alpha_n^{s_n})$. \abstand
  We obtain the fundamental group of $\cu$ as
  $\pi_1(\cu) \isom \pi_1(\cu \setminus \pr^{-1}(\hat{X})) / \langle  \alpha_a^{e_a}  \mid  a \in \hat{X}  \rangle$, which is generated by
  $\{  \alpha_a^{-s} \alpha_b^{s}  \mid  s \in \Z/m\Z,  a,b \in X  \}$. 
  \abstand
  All branch points $a,b \in X$ are connected by a path $(a,v_1,\dots,v_t,b)$ in the spanning tree, so we can write $\alpha_a^{-s} \alpha_b^{s} = (\alpha_a^{-s}\alpha_{v_1}^{s})
  (\alpha_{v_1}^{-s}\alpha_{v_2}^{s})\dots(\alpha_{v_{t-1}}^{-s}\alpha_{v_t}^{s})(\alpha_{v_t}^{-s}\alpha_b^{s})$ and hence we have that
  $\{ \alpha_a^{-s} \alpha_b^{s}  \mid  s \in \Z/m\Z,  (a,b) \in E \}$ generates $\pi_1(\cu)$ and therefore $\homo$. \abstand
  If we choose basepoints $p_0 \in \P^1 \setminus \hat{X}$ for $\pi_1(\P^1 \setminus \hat{X})$ and $P_0 \in \pr^{-1}(p_0)$ for $\pi_1(\cu \setminus \pr^{-1}(\hat{X}))$ and $\pi_1(\cu)$ respectively, then,
  depending on the choice of $P_0$, for all $e = (a,b) \in E$ there exists $l_0 \in \Z/m\Z$ such that $\gamma_e^{(l_0)}$ is homotopic to $\alpha_a^{-1} \alpha_b$ in $\pi_1(\cu,P_0)$.
   In $\homo$ we have that $\alpha_a^{-s}\alpha_b^{s} = ( \alpha_a^{-1}\alpha_b)^s$, so we obtain the other powers by concatenating
  the shifts $\prod_{l = 0}^{s-1} \gamma_e^{(l_0+l)} = (\alpha_a^{-1}\alpha_b)^s$.
  This implies $1 = \prod_{l = 0}^{m-1} \gamma_e^{(l_0+l)} = \prod_{l = 0}^{m-1} \gamma_e^{(l)}$ and
   $$\{  \alpha_a^{-s} \alpha_b^{s}  \mid  s \in \Z/m\Z  \} \subset  \langle  \gamma_e^{(l)}
  \mid  0 \le l < m-1 \rangle,$$ and therefore $\homo = \langle  \Gamma  \rangle$.
  \end{proof}

  \begin{rmk}
  \begin{itemize}
   \item[$\bullet$] For $\delta = 1$, we have that $\# \Gamma = (m-1)(n-1) = 2g$. Therefore, $\Gamma$ is a basis for $\homo$ in that case.
   \item[$\bullet$] In the case $\delta = m$, the point at infinity is not a branch point. Leaving out one finite branch point in the spanning tree results in only $n-2$ edges. Hence, we easily find a subset
   $\Gamma' \subset \Gamma$ such that
   $\# \Gamma' = (m-1)(n-2) = 2g$ and $\Gamma'$ is a basis for $\homo$.
  \end{itemize}
  \end{rmk}

\subsection{Differential forms}\label{subsec:diff_forms}

    The computation of the period matrix and the Abel-Jacobi map requires a basis of $\hd$ as a $\C$-vector space. In this section we provide a basis that only
   depends on $m$ and $n$ and is suitable for numerical integration. \abstand
    Among the meromorphic differentials
    \begin{align*}
 \WM = \left\{  \omega_{i,j}   \right\}_{\substack{1 \le i \le n-1, \\ 1 \le j \le m-1}} \quad \text{with} \quad \omega_{i,j} = \frac{ \dx^i}{i y^j},
  \end{align*}
  there are exactly $g$ that are holomorphic  and they can be found by imposing a simple combinatorial condition on $i$ and $j$.
 The following proposition is basically a more general version of
  \cite[Proposition 2]{CT1996}.

   \bigskip

     \begin{prop}\label{prop:holom_diff}
 Let $\delta = \gcd(m,n)$. The following differentials form  a $\C$-basis of $\hd$:
 \begin{equation*}\label{eq:holm_diff}
   \W =  \left\{  \omega_{i,j} \in \WM  \mid  -mi + jn - \delta \ge 0  \right\}
 \end{equation*}
     \end{prop}
     \begin{proof}
      First we show that the differentials in $\W$ are holomorphic.
      Let $\w_{i,j} = x^{i-1}y^{-j} \dx \in \WM$. We write down the relevant divisors
      \begin{align*}
       \div(x) & = \sum_{k=1}^m \left(0,\zeta^k\sqrt[m]{f(0)}\right) - \frac{m}{\delta} \cdot \sum_{l = 1}^{\delta} P_{\infty}^{(l)}, \\
       \div(y) & = \sum_{k = 1}^n P_k - \frac{n}{\delta} \cdot \sum_{l = 1}^{\delta}  P_{\infty}^{(l)}, \\
       \div(\dx) & = (m-1)\sum_{k = 1}^n P_k - \left(\frac{m}{\delta} + 1\right)\cdot \sum_{l = 1}^{\delta}  P_{\infty}^{(l)}.
      \end{align*}
     Putting together the information, for $P \in \cu$ lying over $x_0 \in \P^1_{\C}$, we obtain
     \begin{align}\label{eq:diff_cases}
      v_P(\omega_{i,j}) & = (i-1) v_P(x) + v_P(\dx)  - j v_P(y)
 \begin{cases}
  \ge 0 \hfill \text{if} \; x_0 \ne x_k,\infty, \\
  = m-1-j \ge 0 \quad \text{if} \; x_0 = x_k, \\
  = \frac{(-mi-\delta+jn)}{\delta} \hfill \text{if}\;  x_0 = \infty.
 \end{cases}
     \end{align}
     We conclude: $\omega_{i,j} \in \WM$ is holomorphic if and only if $\omega_{i,j} \in \W$. \abstand
     Since the differentials in $\W$ are clearly $\C$-linearly independent, it remains to show that
     there are enough of them, i.e. $\#\W = g$.
     \medskip

     Counting the elements in $\W$ corresponds to counting lattice points
     $(i,j) \in \Z^2$ in the trapezoid given by the faces
     \begin{align*}
 1 \le i \le n-1,\\
 1 \le j \le m-1, \\
 i \le \frac{n}{m}j - \frac{\delta}{m}.
     \end{align*}
      \begin{figure}[H]
      \begin{center}
\begin{tikzpicture}
\draw (-0.7,0) -- (8,0) [->]; \draw (8.5,0) node {$j$};
\draw (0,-0.7) -- (0,4) [->]; \draw (0,4.5) node {$i$};
\draw (1,-0.05) -- (1,0.05); \draw (1,-0.4) node {$1$};
\draw (2,-0.05) -- (2,0.05); \draw (2,-0.4) node {$2$};
\draw (3,-0.05) -- (3,0.05); \draw (3,-0.4) node {$3$};
\draw (4,-0.05) -- (4,0.05); \draw (4,-0.4) node {$4$};
\draw (5,-0.05) -- (5,0.05); \draw (5,-0.4) node {$5$};
\draw (6,-0.05) -- (6,0.05); \draw (6,-0.4) node {$6$};
\draw (7,-0.05) -- (7,0.05); \draw (7,-0.4) node {$7$};
\draw (-0.05,1) -- (0.05,1); \draw (-0.4,1) node {$1$};
\draw (-0.05,2) -- (0.05,2); \draw (-0.4,2) node {$2$};
\draw (-0.05,3) -- (0.05,3); \draw (-0.4,3) node {$3$};
\draw[gray] (1,1) -- (1,3);
\draw[gray] (1,1) -- (7,1);
\draw[gray] (7,1) -- (7,3);
\draw[gray] (1,3) -- (7,3);
\draw[gray] (0,-0.5) -- (8,3.5);
\filldraw [gray] (1,1) circle (1pt);
\filldraw [gray] (1,2) circle (1pt);
\filldraw [gray] (1,3) circle (1pt);
\filldraw [gray] (2,1) circle (1pt);
\filldraw [gray] (2,2) circle (1pt);
\filldraw [gray] (2,3) circle (1pt);
\filldraw (3,1) circle (1pt);
\filldraw [gray] (3,2) circle (1pt);
\filldraw [gray] (3,3) circle (1pt);
\filldraw (4,1) circle (1pt);
\filldraw [gray] (4,2) circle (1pt);
\filldraw [gray] (4,3) circle (1pt);
\filldraw (5,1) circle (1pt);
\filldraw (5,2) circle (1pt);
\filldraw [gray] (5,3) circle (1pt);
\filldraw (6,1) circle (1pt);
\filldraw (6,2) circle (1pt);
\filldraw [gray] (6,3) circle (1pt);
\filldraw (7,1) circle (1pt);
\filldraw (7,2) circle (1pt);
\filldraw (7,3) circle (1pt);
\end{tikzpicture}
      \end{center}
    \caption{The points below the line correspond to holomorphic differentials.
    Illustrated is the case $n=4,m=8$, and thus $g = 9$.}
    \label{fig:holom_diff}
\end{figure}
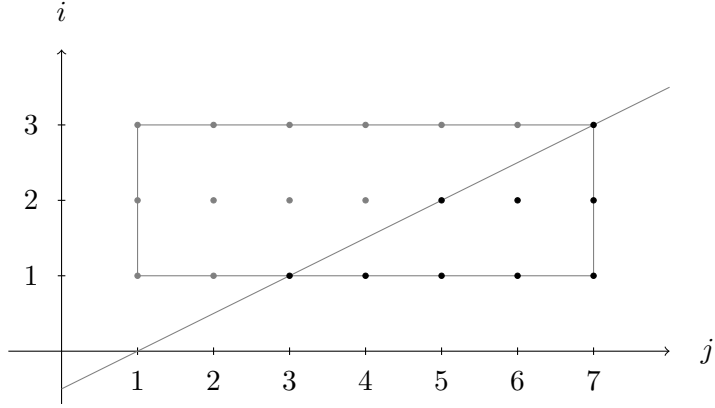

Summing over the vertical lines of the trapezoid, we find the following formula that counts the points.
     \begin{align}\label{eq:r_j}
 \#\W & = \sum_{j = 1}^{m-1} \floor*{\frac{n}{m}j - \frac{\delta}{m}} = \sum_{j = 1}^{m-1} \frac{nj-\delta-r_j}{m} =
  \frac{n}{m} \sum_{j=1}^{m-1} j - \frac{m-1}{m}\delta - \frac{1}{m} \sum_{j=1}^{m-1} r_j,
     \end{align}
      where $r_j = nj - \delta  \bmod m$.

The desired equality  $\#\W = \frac{1}{2}((n-1)(m-1)-\delta+1) = g$ immediately follows from
\begin{lemma}
      \begin{equation*}\label{eq:r_j2}
       \sum_{j=1}^{m-1} r_j = \frac{1}{2}(m^2 - (\delta+2)m + 2\delta).
      \end{equation*}
\end{lemma}
\begin{proof}\let\qed\relax
      Let $l := \frac{m}{\delta}$. First we note that $r_j = r_{j+l}$:
      $$r_{j+l} = n(j+l) - \delta  \bmod m = nj + \frac{n}{\delta}m - \delta  \bmod m =  nj - \delta  \bmod m =  r_j,$$
      and hence
      \begin{align}\label{eq:r_j3}
       \sum_{j=1}^{m-1} r_j = \delta \cdot \sum_{j=1}^{l} r_j - r_m = \delta \cdot \sum_{j=1}^{l} r_j - (-\delta + m).
      \end{align}
      Furthermore, $r_j$ can be written as a multiple of $\delta$:
      $$r_j = \delta \left(\frac{n}{\delta}j - 1\right)  \bmod m.$$
      From $\gcd(\frac{n}{\delta},l) = 1$ we conclude $\left\{  \frac{n}{\delta}j - 1  \bmod l  \mid  1 \le j \le l  \right\} = \{  0,\dots,l-1  \}$.
      Therefore,
      \begin{align}\label{eq:r_j4}
       \sum_{j = 1}^l r_j = \sum_{j = 0}^{l-1} \delta j = \delta \cdot \frac{l(l-1)}{2},
      \end{align}
      and thus \eqref{m-eq:r_j3} and \eqref{m-eq:r_j4} imply
      $$\sum_{j=1}^{m-1} r_j = \delta \cdot \sum_{j=1}^{l} r_j + \delta - m = \delta^2 \cdot \frac{l(l-1)}{2} + \delta - m = \frac{1}{2}(m^2 - (\delta+2)m + 2\delta).$$
\end{proof}
\end{proof}

\begin{rmk}
    Note that from \eqref{m-eq:diff_cases} it follows that the meromorphic differentials in $\WM$ are homolorphic at all finite points.
\end{rmk}

  \section{Strategy for the period matrix}\label{sec:strat_pm}

  In this section we present our strategy to obtain period matrices $\OC, \OA, \OB$ and $\tau$ as defined in
  \S \ref{m-subsec:bases_matrices}. Although this paper is not restricted to the
  case $\gcd(m,n) = 1$, we will briefly assume it in this paragraph to simplify notation.

  The main ingredients were already described in
  Section \ref{m-sec:se_curves}: we integrate the holomorphic differentials in $\W$ (\S \ref{m-subsec:diff_forms})
  over the cycles in $\Gamma$ (\S \ref{m-subsec:cycles_homo}) using numerical integration (\S \ref{m-subsec:de_int}), which results in a period matrix (\S \ref{m-subsec:comp_of_periods})
  \begin{equation*}
      \label{eq:OC}
    \OC = \left( \int_{\gamma} \w \right)_{\substack{\w \in \W, \\ \gamma \in \Gamma}} \in \C^{g \times 2g}.
  \end{equation*}
  The matrices $\OA$ and $\OB$ require a symplectic basis of $\homo$.
  So, we compute the intersection pairing on $\Gamma$, as explained in Section \ref{m-sec:intersections}, which results in a
  intersection matrix $K_{\Gamma} \in \Z^{2g \times 2g}$.
  After computing a symplectic base change $S \in \GL(\Z,2g)$ for $K_{\Gamma}$ (\S \ref{m-subsec:symp_basis}), we obtain a big period matrix
  \begin{equation}
      \label{eq:OAOB}
      (\OA,\OB) = \OC S,
  \end{equation}
   and finally a small period matrix in the Siegel upper half-space
  \begin{equation}
      \label{eq:tau}
   \tau = \OA^{-1} \OB \in \mathfrak{H}_g.
  \end{equation}

  \bigskip

  \subsection{Periods of elementary cycles}\label{subsec:comp_of_periods}

  The following theorem provides a formula for computing the periods of the curve.
  It relates integration of differential forms on the curve to numerical integration in $\C$.

  Note that the statement is true for all differentials in $\WM$, not just the holomorphic ones.
  We continue to use the notation from Section \ref{m-sec:se_curves}.

  \begin{thm}\label{thm:periods}
   Let $\gamma_e^{(l)} \in \Gamma$ be a shift of an elementary cycle corresponding
   to an edge $e = (a,b) \in E$. Then, for all differentials $\w_{i,j} \in \WM$, we have
   \begin{equation}\label{eq:periods}
      \int_{\gamma_e^{(l)}} \w_{i,j}
      =  \zeta^{-lj} (1-\zeta^{-j}) C_{a,b}^{-j} \left(\frac{b-a}{2}\right)^i \int_{-1}^1 \frac{\varphi_{i,j}(u)}{(1-u^2)^{\frac{j}{m}}}  \du,
   \end{equation}
   where
   \begin{equation*}
    \varphi_{i,j}  = \left(u+\frac{b+a}{b-a}\right)^{i-1} \ytab(u)^{-j}
   \end{equation*}
   is holomorphic in a neighbourhood $\epsilon_{a,b}$ of $[-1,1]$.
  \end{thm}
  \begin{proof}
    By the definition in \eqref{m-eq:def_cyabl} we can write $\gamma_e^{(l)} = \gamma_{[a,b]}^{(l)} \cup \gamma_{[b,a]}^{(l+1)}$. Hence we split up the integral and compute
    \begin{align*}\label{eq:thm_periods_2}
     \int_{\gamma_{[a,b]}^{(l)}} \w_{i,j}  & =  \int_{\gamma_{[a,b]}^{(l)}} \frac{x^{i-1}}{y^j}  \dx  =  \zeta^{-lj} \int_a^b \frac{x^{i-1}}{\yab(x)^j}  \dx \\  & =
     \zeta^{-lj} C_{a,b}^{-j}   \int_a^b \frac{x^{i-1}}{\ytab(\uab(x))^j (1-\uab(x)^2)^{\frac{j}{m}}} \dx.
     \intertext{
  Applying the transformation $x \mapsto \xab(u)$ introduces the derivative $\dx = \left(\frac{b-a}{2}\right) \du$ yields
  }
   \int_{\gamma_{[a,b]}^{(l)}} \w_{i,j} & = \zeta^{-lj} C_{a,b}^{-j} \left(\frac{b-a}{2}\right) \int_{-1}^1 \frac{\xab(u)^{i-1}}{\ytab(u)^j (1-u^2)^{\frac{j}{m}}}  \du \\ & =
    \zeta^{-lj} C_{a,b}^{-j} \left(\frac{b-a}{2}\right)^i \int_{-1}^1 \frac{\left(u+\frac{b+a}{b-a}\right)^{i-1}}{\ytab(u)^j (1-u^2)^{\frac{j}{m}}}  \du
  \end{align*}
  Similarly, we obtain
  \begin{equation*}
        \int_{\gamma_{[b,a]}^{(l+1)}} w_{i,j}  =  -\zeta^{-j} \int_{\gamma_{[a,b]}^{(l)}} w_{i,j}.
  \end{equation*}
  By Proposition \ref{m-prop:yab}\,, $\ytab$ is holomorphic and has no zero on $\epsilon_{a,b}$, therefore \\ $\varphi_{i,j}  = \left(u+\frac{b+a}{b-a}\right)^{i-1} \ytab(u)^{-j}$
  is holomorphic on $\epsilon_{a,b}$.
  \end{proof}

   \subsection{Numerical integration}\label{subsec:numerical_integration}

   In order to compute a period matrix $\OC$ the only integrals that have to
   be numerically evaluated are
   the \emph{elementary integrals}
   \begin{equation}\label{eq:elem_ints}
       \int_{-1}^1 \frac{\varphi_{i,j}(u)}{(1-u^2)^{\frac{j}{m}}} \du
   \end{equation}
   for all $\w_{i,j} \in \W$ and $e \in E$. By Theorem \ref{m-thm:periods}, all the periods in $\OC$ are
   then obtained by multiplication of elementary integrals with constants.

 As explained in \S \ref{m-subsec:real_mult}, the actual computations will be done on integrals of the form
\begin{equation}
    \label{eq:elem_num_int}
    I_{a,b}(i,j) = \int_{-1}^1\frac{u^{i-1}\du}{(1-u^2)^{\frac jm}\ytab(u)^j}
\end{equation}
(that is, replacing $(u+\frac{b+a}{b-a})^{i-1}$ by $u^{i-1}$ in the numerator of $\phi_{i,j}$),
the value of elementary integrals being recovered by the polynomial shift
\begin{equation}
    \label{eq:polshift}
    \int_{-1}^1\frac{\varphi_{i,j}(u)}{(1-u^2)^{\frac jm}}\du
    = \sum_{l=0}^{i-1} {i-1 \choose l} \left(\frac{b+a}{b-a}\right)^{i-1-l} I_{a,b}(l,j).
\end{equation}

The rigorous numerical evaluation of \eqref{m-eq:elem_num_int} is
adressed in Section~\ref{m-sec:numerical_integration}: for any edge $(a,b)$,
Theorems~\ref{m-thm:de_int} and \ref{m-thm:gc_int} provide
explicit schemes allowing to attain any prescribed precision.

  \subsection{Minimal spanning tree}\label{subsec:spanning_tree}

  From the a priori analysis of all numerical integrals $I_{a,b}$ along
  the interval $[a,b]$, we choose an optimal set of edges forming a spanning tree as follows:
  \begin{itemize}
      \item
   Consider the complete graph on the set of finite branch points $G' = (X,E')$ where
   $E' = \{  (a,b)  \mid  a,b \in X \}$.
      \item
   Each edge $e = (a,b) \in E'$ gets assigned a capacity $r_e$ that indicates
   the cost of numerical integration along the interval $[a,b]$.
   \item
   Apply a standard `maximal-flow' algorithm from graph theory, based on a greedy approach.
   This results in a spanning tree $G = (X,E)$, where $E \subset E'$ contains the $n-1$ best edges
   for integration that connect all vertices without producing cycles.
  \end{itemize}

   Note that the integration process is most favourable
   between branch points that are far away from the others
   (this notion is made explicit in Section \ref{m-sec:numerical_integration}).

  \subsection{Symplectic basis}\label{subsec:symp_basis}

  By definition, a big period matrix $(\OA,\OB)$ requires integration along
  a symplectic basis of $\homo$. In \S \ref{m-subsec:cycles_homo} we gave
  a generating set $\Gamma$ for $\homo$, namely
  \begin{equation*}
    \Gamma = \left\{  \gamma_{e}^{(l)}  \mid  0 \le l < m-1,  e \in E  \right\},
  \end{equation*}
  where $E$ is the spanning tree chosen above. This generating set is
  in general not a (symplectic) basis.

  We resolve this by computing the intersection pairing on $\Gamma$, that
  is all intersections
  $\gamma_{e}^{(k)} \circ \gamma_{f}^{(l)}\in\{ 0,\pm 1\}$ for $e,f\in E$ and $k,l\in \set{0,\dots m-1}$,
  as explained in Section \ref{m-sec:intersections}.

  The resulting intersection matrix $K_{\Gamma}$
   is a skew-symmetric matrix of dimension \\ $(n-1)(m-1)$
   and has rank $2g$.

   Hence, we can apply an algorithm, based on
    \cite[Theorem 18]{KB2002}, that outputs a symplectic basis
   for $K_{\Gamma}$ over $\Z$, i.e. a unimodular matrix base change matrix $S$ such that
  $$S^T K_{\Gamma}  S = J, \quad \text{where} \quad J = \begin{pmatrix} 0 & I_g & 0 \\ -I_g & 0 & 0 \\ 0 & 0 & 0_{\delta-1} \end{pmatrix}.$$

  The linear combinations of periods given by the first $2g$ columns of $\OC S$ then correspond to a symplectic homology basis
  \begin{equation*}\label{eq:symp_basis}
   (\OA,\OB, 0_{\delta-1}) = \OC S,
  \end{equation*}
  whereas the last $\delta-1$ columns are zero and can be ignored, as they correspond to the dependent cycles
  in $\Gamma$ and contribute nothing.

  \section{Intersections}\label{sec:intersections}

   Let $(a,b)$ and $(c,d)$ be two edges of the spanning tree $E$.
  The formulas in Theorem \ref{m-thm:intsec_numb} allow to compute the intersection
  between shifts of elementary cycles $\left(\cyab^{(k)} \circ \cycd^{(l)}\right)$.

  Note that by construction of the spanning tree,
  we can restrict the analysis to intersections $\left(\cyab^{(k)} \circ \cycd^{(l)}\right)$
  such that $c$ is either $a$ or $b$.

   \begin{thm}[Intersection numbers]\label{thm:intsec_numb}
      Let $(a,b),(c,d) \in E$. The intersections of the corresponding cycles $\cyabk, \cycdl \in \Gamma$ are given by
      \begin{equation*}
          \left(\cyab^{(k)} \circ \cycd^{(l)}\right)
          = \begin{cases}
              1  &\text{ if } l-k \equiv s_+ \bmod m,\\
              -1 &\text{ if } l-k \equiv s_- \bmod m,\\
              0 &\text{ otherwise,}
          \end{cases}
      \end{equation*}
      where $s_+$, $s_-$ are given by the following table, which covers all
      cases occurring in the algorithm
      \begin{center}
          \normalfont
      \begin{tabular}{cccc}
          \toprule
          & case & $s_+$ & $s_-$ \\
          \midrule
         (i) & $a=c$ and $b=d$ & $1$ & $-1$ \\
         (ii) & $b=c$ & $-s_b$ & $1-s_b$ \\
         (iii) & $a=c$ and $\rho>0$ & $1-s_a$ & $-s_a$ \\
         (iv) & $a=c$ and $\rho<0$ & $-s_a$ & $-1-s_a$\\
         (v) & $\set{a,b}\cap\set{c,d}=\varnothing$ & \multicolumn{2}{c}{no intersection} \\
          \bottomrule
      \end{tabular}
      \end{center}
      and where $s_x \in \Z$ for $x\in\set{a,b}$ is given by
      \begin{equation*}\label{eq:s_x}
	    s_x := \frac{1}{2\pi}\left( \rho + m \cdot \arg \left( \frac{C_{c,d} \ytcd(x)}{C_{a,b}\ytab(x)} \right)
	    \right)
      \end{equation*}
     and
      \begin{equation*}
          \rho = \arg \left( \frac{b-a}{d-c} \right) + \delta_{b=c}\pi.
      \end{equation*}
 \end{thm}

 \begin{rmk}
     Note that the intersection matrix $K_{\Gamma}$ is composed of
  $(n-1)^2$ blocks of dimension $m-1$, each block corresponding to the intersection
  of shifts of two elementary cycles in the spanning tree. It is very sparse.
 \end{rmk}

 \bigskip
 The proof of Theorem \ref{m-thm:intsec_numb} is contained in the following exposition.
 \bigskip

  Consider
  two cycles $\cyab^{(k)},\cycd^{(l)} \in \Gamma$ and
  recall from \S \ref{m-subsec:roots_branches} their definition
   \begin{align*}
      \cyabk & = \{  (x,\zeta^k \yab(x))  \mid  x \in [a,b]  \} \cup \{  (x,\zeta^{k+1} \yab(x))  \mid  x \in [b,a]  \}, \\
      \cycdl & = \{  (x,\zeta^l \ycd(x))  \mid  x \in [c,d]  \} \cup \{  (x,\zeta^{l+1} \ycd(x))  \mid  x \in [c,d]  \},
   \end{align*}
  where $\zeta^k\yab(x),\zeta^l\ycd(x)$ are branches of $\cu$ that are analytic on open sets $\vab$ and $\vcd$
  (see Figure \ref{m-fig:set_vab}) respectively.

  \begin{proof}
  From the definition we see that $\cyabk \cap \cycdl = \varnothing$, whenever $[a,b] \cap [c,d] = \varnothing$. For edges in a spanning tree
  this is equivalent to $\{a,b\} \cap \{c,d\} = \varnothing$, thus
  proving (v).
  \end{proof}

   Henceforth, we can assume $\{a,b\} \cap \{c,d\} \ne \varnothing$. In order to prove (i)-(iv) we have to introduce some machinery. Since the $\yab(x),\ycd(x)$ are branches of $\cu$,
   on the set $\C \setminus X$
   we can define the \emph{shifting function}
   $s(x)$, that takes values in $\Z/m\Z$, implicitly via
  \begin{equation}\label{eq:shift_func}
   \zeta^{s(x)} = \frac{\ycd(x)}{\yab(x)}.
   \end{equation}
  Naturally, \eqref{eq:shift_func} extends to the other analytic branches via
  \begin{equation*}
   \zeta^{s(x)+l-k} = \frac{\zeta^l\ycd(x)}{\zeta^k\yab(x)}.
   \end{equation*}
   We can now define the non-empty, open, disconnected set
   \begin{equation*}
    V := \vab \cap \vcd \subset \C \setminus X.
   \end{equation*}
 The shifting function $s(x)$ is well-defined on $V$ and, since $\yab(x)$ and $\ycd(x)$ are
  both analytic on $V$, $s(x)$ is constant on its
  connected components.

  In \S \ref{m-subsec:riemann_surface} we established that multiplication of a branch by $\zeta$ corresponds to moving
  one sheet up on the Riemann surface.
  We can interpret the value of the shifting function geometrically as
  $\cycdl$ running $s(\xt)+l-k$ sheets above $\cyabk$ at a point $\xt \in V$.

 This can be used to determine the intersection number in the following way. We deform the cycles homotopically
 such that
 \begin{equation*}
   \pr\left(\cyabk\right) \cap \pr\left(\cycdl\right) = \{ \xt \}  \text{ for some $\xt \in V$.}
 \end{equation*}
 Consequently, the cycles can at most intersect at the
 points in the fiber above $\xt$, i.e.
 \begin{equation*}
  \cyab^{(k)} \cap \cycd^{(l)} \subset \text{pr}_x^{-1}(\xt).
 \end{equation*}
 Note that, by definition, any cycle in $\Gamma$ only runs on two neighbouring sheets, which already implies
 \begin{equation*}
   \left(\cyab^{(k)} \circ \cycd^{(l)}\right) = 0, \text{ if $s(\xt)+l-k \not\in \{-1,0,1\}$.}
 \end{equation*}
  In the other cases we can determine the
  sign of possible intersections by taking into account the orientation of the cycles.

 We continue the proof with case (i):
 Here we have $[a,b] = [c,d]$. Trivially, $\left( \cyabk \circ \cyabk \right) = 0$ holds. For $k \ne l$ we deform the cycles such that they only intersect above
 $\xt = \frac{b+a}{2} \in  \vab= V$.
  We easily see that $s(\xt) = 0$ and therefore $s(\xt) + l - k = l - k$. The remaining non-trivial cases
 ($l = k \pm 1$), are shown in Figure \ref{m-fig:int_self_shift} below where
   the cycles $\cyabk$ (black),
      $\cyab^{(k+1)}$ (red) and $\cyab^{(k-1)}$ (green) are illustrated. 
    \begin{figure}[H]
      \begin{center}
   \scalebox{0.8}{\begin{tikzpicture}
     \draw [densely dashed] (-3,3) -- (-1.8,3); 
     \draw [densely dashed] (3,3) -- (1.8,3); 
     \draw (-1.8,3) circle (1.3pt); 
     \draw (1.8,3) circle (1.3pt);
     \draw [red] (-2.5,3) .. controls (-2.5,2.6) and (-1,2.4) .. (0,3) [halfarrow2];
     \draw [red] (0,3) .. controls (1,3.6) and (2.5,3.4) .. (2.5,3) [halfarrow2];

     \draw [densely dashed] (-3,1) -- (-1.8,1); 
     \draw [densely dashed] (3,1) -- (1.8,1); 
     \draw (-1.8,1) circle (1.3pt); 
     \draw (1.8,1) circle (1.3pt);
     \draw[red] (-2.5,1) .. controls (-2.5,1.6) and (-1,1.4) .. (0,1) [halfarrow1];
     \draw[red] (0,1) .. controls (1,0.6) and (2.5,0.4) .. (2.5,1) [halfarrow1];
     \draw (-2.5,1) .. controls (-2.5,0.6) and (-1,0.4) .. (0,1) [halfarrow2];
     \draw (0,1) .. controls (1,1.6) and (2.5,1.4) .. (2.5,1) [halfarrow2];

    \draw (0,1) node[cross=3pt,black]{};
    \draw (0,1.3) node {$+1$};
    \draw (-1.8,0) node {$a$};  
    \draw (0,0) node {$\text{pr}_x^{-1}(\xt)$}; 
    \draw (1.8,0) node {$b$};

    \draw [densely dashed] (-3,-1) -- (-1.8,-1); 
    \draw [densely dashed] (3,-1) -- (1.8,-1); 
    \draw (-1.8,-1) circle (1.3pt); 
    \draw (1.8,-1) circle (1.3pt);
    \draw (-2.5,-1) .. controls (-2.5,-0.4) and (-1,-0.6) .. (0,-1) [halfarrow1];
    \draw (0,-1) .. controls (1,-1.4) and (2.5,-1.6) .. (2.5,-1) [halfarrow1];
    \draw [green] (-2.5,-1) .. controls (-2.5,-1.4) and (-1,-1.6) .. (0,-1) [halfarrow2];
    \draw [green] (0,-1) .. controls (1,-0.4) and (2.5,-0.6) .. (2.5,-1) [halfarrow2];

    \draw (0,-1) node[cross=3pt,black]{};
    \draw (0,-1.3) node {$-1$};

    \draw [densely dashed] (-3,-3) -- (-1.8,-3); 
    \draw [densely dashed] (3,-3) -- (1.8,-3); 
    \draw (-1.8,-3) circle (1.3pt); 
    \draw (1.8,-3) circle (1.3pt);
    \draw[green] (-2.5,-3) .. controls (-2.5,-2.4) and (-1,-2.6) .. (0,-3) [halfarrow1];
    \draw[green] (0,-3) .. controls (1,-3.4) and (2.5,-3.6) .. (2.5,-3) [halfarrow1];
\end{tikzpicture}
}
      \end{center}
    \caption{Intersections of self-shifts.}
    \label{fig:int_self_shift}
\end{figure}
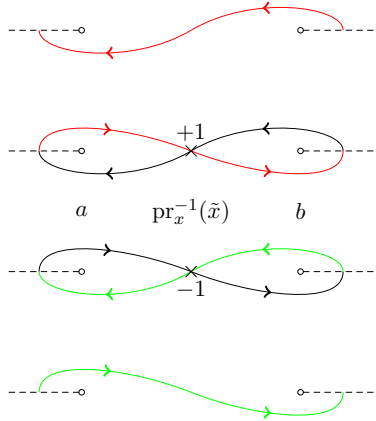
  We see that, independently of $s(\xt)$, $s_+ = (k+1)-k = 1$ and $s_- = (k-1)-k = -1$ are as claimed.

  For (ii)-(iv) we have that $[a,b] \cap [c,d] = \{c \}$, where $c$ is either $a$ or $b$. Unfortunately, in these cases $s(c)$ is not
  well-defined.

 Instead, we choose a point $\xt \in \C \setminus X$ on the bisectrix of
  $[a,b]$ and $[c,d]$ that is close enough to $c$ such that $[\xt,c[ \subset V = \vab \cap \vcd$
  (see Figure \ref{m-fig:set_v_both} below), and where
  \begin{equation}\label{eq:sxt}
   s(\xt) = \frac{m}{2\pi} \arg \left( \frac{\ycd(\xt)}{\yab(\xt)} \right).
  \end{equation}
    \begin{figure}[H]
      \begin{center}
   \scalebox{.9}{
       \begin{tikzpicture}
    
  \draw (-3.3,0) node {$a$};  \draw (-0.3,0) node {$b$}; \draw (-1.8,0) node {$V_{a,b}$};
  \draw (-3.6,0) circle (0.8pt); 
  \draw  (0,0) circle (0.8pt);
  \draw [densely dashed] (-4.8,0) -- (-3.6,0);   
  \draw [densely dashed] (1.5,0) -- (0,0);
  \draw (-4.2,0.05) .. controls (-3.8,1.5) and (0.2,1.5) .. (0.6,0.05);
  \draw (-4.2,-0.05) .. controls (-3.8,-1.5) and (0.2,-1.5) .. (0.6,-0.05);
  \draw (-4.2,-0.05) -- (-3.5,-0.05); \draw (0.6,-0.05) -- (-0.1,-0.05);
  \draw (-4.2,0.05) -- (-3.5,0.05); \draw (0.6,0.05) -- (-0.1,0.05);
  \draw (-3.5,0.05) -- (-3.5,-0.05); \draw (-0.1,0.05) -- (-0.1,-0.05);
    
  \draw (1.8,3.117) circle (0.8pt); 
  \draw (1.8,2.8) node {$d$}; \draw (1.0,1.5) node {$V_{b,d}$};
  \draw [densely dashed] (1.8,3.117) -- (2.4,4.156); 
  \draw [densely dashed] (0,0) -- (-0.8,-1.386);
  \draw [rotate=240,shift={(0,0)}] (-4.2,0.05) .. controls (-3.8,1.5) and (0.2,1.5) .. (0.6,0.05);
  \draw [rotate=240,shift={(0,0)}] (-4.2,-0.05) .. controls (-3.8,-1.5) and (0.2,-1.5) .. (0.6,-0.05);
  \draw [rotate=240,shift={(0,0)}] (-4.2,-0.05) -- (-3.5,-0.05);  \draw [rotate=240,shift={(0,0)}] (0.6,-0.05) -- (-0.1,-0.05);
  \draw [rotate=240,shift={(0,0)}](-4.2,0.05) -- (-3.5,0.05); \draw [rotate=240,shift={(0,0)}] (0.6,0.05) -- (-0.1,0.05);
  \draw [rotate=240,shift={(0,0)}] (-3.5,0.05) -- (-3.5,-0.05); \draw [rotate=240,shift={(0,0)}] (-0.1,0.05) -- (-0.1,-0.05);
       
  \draw [dotted] (-1,1.732) -- (0.75,-1.3);
  \draw (-0.35,0.606) circle (0.8pt); \draw (-0.1,0.6) node {$\tilde{x}$};


  \draw (5.1,0.3) node {$a$};
  \draw (7.3,2.3) node {$b$};
  \draw (7.3,-2.3) node {$d$};
  \draw (6.5,1.5) node {$V_{a,b}$};
  \draw (6.5,-1.5) node {$V_{a,d}$};
  \draw (5,0) circle (0.8pt);
  \draw (7.545,-2.545) circle (0.8pt); 
  \draw (7.545,2.545) circle (0.8pt);
  \draw [densely dashed] (5,0) -- (5-0.848,0.848);   
  \draw [densely dashed] (5,0) -- (5-0.848,-0.848);
  \draw [densely dashed] (7.545,-2.545) -- (8.39,-3.39);
  \draw [densely dashed] (7.545,2.545) -- (8.39,3.39);
  \draw [dotted] (5-1,0) -- (8,0);
  \draw (5.5,0) circle (0.8pt); \draw (5.7,0) node {$\tilde{x}$};
  
  \draw [shift={(5,0)},rotate=225] (0.6,0.05) .. controls (0.2,1.5) and (-3.8,1.5) .. (-4.2,0.05);
  \draw [shift={(5,0)},rotate=225] (0.6,-0.05) .. controls  (0.2,-1.5) and (-3.8,-1.5) .. (-4.2,-0.05);
  \draw [shift={(5,0)},rotate=225] (-3.5,-0.05) -- (-4.2,-0.05); \draw [shift={(5,0)},rotate=225] (-0.1,-0.05) -- (0.6,-0.05);
  \draw [shift={(5,0)},rotate=225] (-3.5,0.05) -- (-4.2,0.05); \draw [shift={(5,0)},rotate=225] (-0.1,0.05) -- (0.6,0.05);
  \draw [shift={(5,0)},rotate=225] (-3.5,-0.05) -- (-3.5,0.05);  \draw [shift={(5,0)},rotate=225] (-0.1,-0.05) -- (-0.1,0.05);
  
  \draw [shift={(5,0)},rotate=-225] (0.6,0.05) .. controls (0.2,1.5) and (-3.8,1.5) .. (-4.2,0.05);
  \draw [shift={(5,0)},rotate=-225] (0.6,-0.05) .. controls  (0.2,-1.5) and (-3.8,-1.5) .. (-4.2,-0.05);
  \draw [shift={(5,0)},rotate=-225] (-3.5,-0.05) -- (-4.2,-0.05); \draw [shift={(5,0)},rotate=-225] (-0.1,-0.05) -- (0.6,-0.05);
  \draw [shift={(5,0)},rotate=-225] (-3.5,0.05) -- (-4.2,0.05); \draw [shift={(5,0)},rotate=-225] (-0.1,0.05) -- (0.6,0.05);
  \draw [shift={(5,0)},rotate=-225] (-3.5,-0.05) -- (-3.5,0.05);  \draw [shift={(5,0)},rotate=-225] (-0.1,-0.05) -- (-0.1,0.05);
\end{tikzpicture}
}
      \end{center}
     \vspace{-1cm}
    \caption{The set $V = \vab \cap \vcd$ for $b=c$ (left) and $a=c$ (right).}
    \label{fig:set_v_both}
   \end{figure}
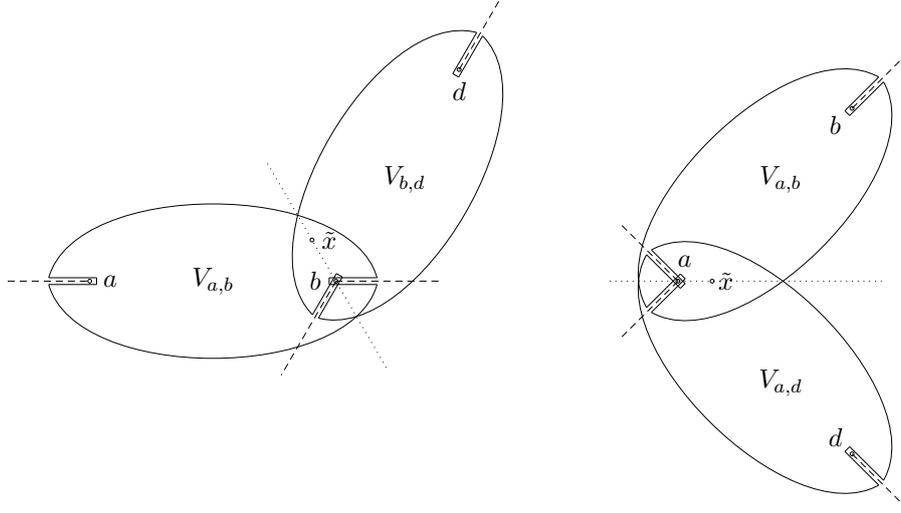

   Case (ii): \\
   In this case we have $b = c$. Choosing $\xt$ on the upper bisectrix (as shown in Figure \ref{m-fig:set_v_both})
   and computing $s(\xt)$ with \eqref{m-eq:sxt} makes it possible to determine the intersection numbers geometrically.

   Figure \ref{m-fig:int_b=c} shows the non-trivial cases $s(\xt) + l - k \in \{-1,0,1\}$.
   There the cycles $\cyab^k$ (black), $\cybd^{k-s(\xt)}$ (gray), $\cybd^{k-s(\xt) + 1 }$ (green) and
    $\cybd^{k-s(\xt)-1}$ (red) are illustrated.
    \begin{figure}[H]
      \begin{center}
   \scalebox{0.8}{
\begin{tikzpicture}

  \draw (-3.6,2) node {$a$};  \draw (0,2) node {$b$};  \draw (1.8,5.117) node {$d$};

    \draw (-0.35,4.606) node[cross=3pt,black,rotate=80]{};
    \draw (-0.25,5) node {$-1$};    
    \draw (-0.35,0.606) node[cross=3pt,black,rotate=80]{};
    \draw (0.0,0.5) node {$+1$}; 

    \draw (-3.6,8) circle (0.8pt); 
    \draw (0,8) circle (0.8pt);   
    \draw (1.8,11.117) circle (0.8pt); 
    \draw [densely dashed] (-4.8,8) -- (-3.6,8);
    \draw [densely dashed] (0,8) -- (1.2,8); 
    \draw [densely dashed] (0,8) -- (-0.6,8-1.04);
    \draw [densely dashed] (1.8,11.117) -- (2.4,12.156);
    \draw [green] (1.95,11.377) .. controls (0.2,12) and (1.5,6.8) .. (-0.4,7.3) [halfarrow1];

    \draw (-3.6,4) circle (0.8pt); 
    \draw (0,4) circle (0.8pt);   
    \draw (1.8,7.117) circle (0.8pt);   
    \draw [densely dashed] (-4.8,4) -- (-3.6,4);
    \draw [densely dashed] (0,4) -- (1.2,4); 
    \draw [densely dashed] (0,4) -- (-0.6,4-1.04);
    \draw [densely dashed] (1.8,7.117) -- (2.4,8.156);
    \draw (-4.1,4) .. controls (-4,2.5) and (0.4,5.95) .. (0.45,4) [halfarrow2];
    \draw [green] (1.95,7.377) .. controls (2.8,7) and (2,5.5) .. (-0.35,4.606) [halfarrow2];
    \draw [green] (-0.35,4.606) arc(110:245:0.7cm) [halfarrow2];
    \draw [gray] (1.95,7.377) .. controls (0.2,8) and (1.5,2.8) .. (-0.4,3.3) [halfarrow1];

    \draw (-3.6,0) circle (0.8pt); 
    \draw (0,0) circle (0.8pt);   
    \draw (1.8,3.117) circle (0.8pt); 
    \draw [densely dashed] (-4.8,0) -- (-3.6,0);
    \draw [densely dashed] (0,0) -- (1.2,0); 
    \draw [densely dashed] (0,0) -- (-0.6,-1.04);
    \draw [densely dashed] (1.8,3.117) -- (2.4,4.156);
    \draw (-4.1,0) .. controls (-4,1) and (0.4,1.5) .. (-0.45,0.4) [halfarrow1];
    \draw [gray] (1.95,3.377) .. controls (2.8,3) and (2,1.5) .. (-0.35,0.606) [halfarrow2];
    \draw (-0.45,0.4) arc(150:342:0.5cm) [halfarrow1];
    \draw [gray] (-0.35,0.606) arc(110:245:0.7cm) [halfarrow2];
    \draw [red] (1.95,3.377) .. controls (0.2,4) and (1.5,-1.2) .. (-0.4,-0.7) [halfarrow1];

    \draw [black] (-3.6,-4) circle (0.8pt); 
    \draw [black] (0,-4) circle (0.8pt);   
    \draw [black] (1.8,-0.883) circle (0.8pt); 
    \draw [densely dashed] (-4.8,-4) -- (-3.6,-4);
    \draw [densely dashed] (0,-4) -- (1.2,-4); 
    \draw [densely dashed] (0,-4) -- (-0.6,-5.04);
    \draw [densely dashed] (1.8,-4+3.117) -- (2.4,0.156);
    \draw [red] (-0.35,-3.394) arc(110:245:0.7cm) [halfarrow2];
    \draw [red] (1.95,-0.623) .. controls (2.8,-1) and (2,-2.5) .. (-0.35,-3.394) [halfarrow2];
 
\end{tikzpicture}
}
      \end{center}
    \caption{Intersections for $b=c$.}
    \label{fig:int_b=c}
   \end{figure}
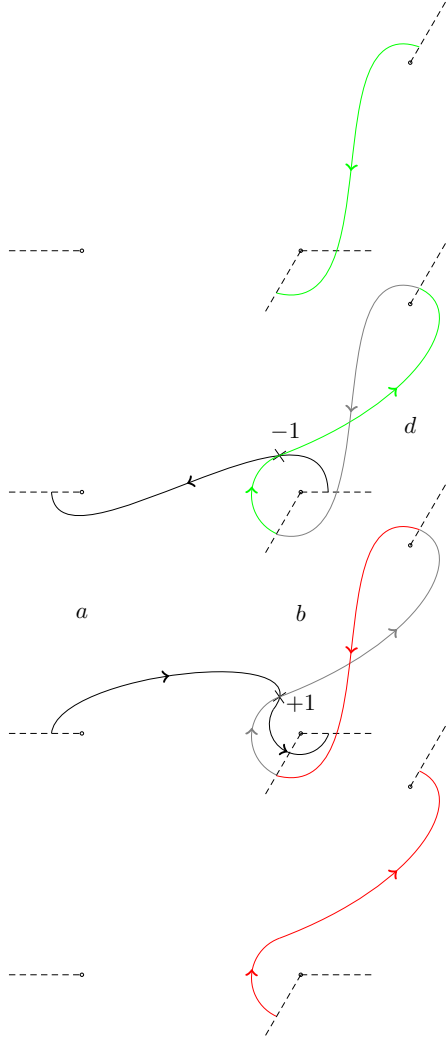
    By Lemma \ref{lemma:sxt=sx} (1) we have $s(\xt) \equiv s_b$, which implies (as claimed)
    \begin{align*}
    & s_+ \equiv k-s(\xt)-k \equiv -s_b \mod m, \\
    & s_- \equiv k-s(\xt)+1-k \equiv 1-s_b \mod m.
    \end{align*}

 Cases (iii) and  (iv): \\
 In these cases we have $a = c$. We choose $\xt$ on the inner bisectrix (as shown in Figure \ref{m-fig:set_v_both})
   and compute $s(\xt)$ with \eqref{m-eq:sxt}.

 For $\varphi = \arg\left(\frac{b-a}{d-c}\right) > 0$, the non trivial cases, i.e. $s(\xt) + l - k \in \{-1,0,1\}$, are shown in Figure \ref{m-fig:int_a=c}
  We illustrate the cycles $\cyab^{(k)}$ (black), $\cyad^{(k-s(\xt))}$ (gray), $\cyad^{(k-s(\xt) + 1)}$ (green) and
    $\cyad^{k-s(\xt)-1}$ (red).
  \begin{figure}[H]
      \begin{center}
   \scalebox{0.8}{
\begin{tikzpicture}


    \draw (5.5,0) node[cross=3pt,black,rotate=60]{};
    \draw (5.6,0.4) node {$-1$};    
    \draw (10.5,0) node[cross=3pt,black,rotate=60]{};
    \draw (11,0) node {$+1$};
 
\draw (0,0) circle (0.8pt);
\draw (2.545,-2.545) circle (0.8pt); 
\draw (2.545,2.545) circle (0.8pt);
\draw [densely dashed] (0,0) -- (-0.848,0.848);   
\draw [densely dashed] (0,0) -- (-0.848,-0.848);
\draw [densely dashed] (2.545,-2.545) -- (3.39,-3.39);
\draw [densely dashed] (2.545,2.545) -- (3.39,3.39);    

\draw [red]  (-0.353,0.353) .. controls (0.8,1.3) and (1.3,-3.7) .. (2.757,-2.757) [halfarrow1];

\draw (5,0) circle (0.8pt);
\draw (7.545,-2.545) circle (0.8pt); 
\draw (7.545,2.545) circle (0.8pt);
\draw [densely dashed] (5,0) -- (5-0.848,0.848);   
\draw [densely dashed] (5,0) -- (5-0.848,-0.848);
\draw [densely dashed] (7.545,-2.545) -- (8.39,-3.39);
\draw [densely dashed] (7.545,2.545) -- (8.39,3.39); 

\draw [black] (5.5,0) arc (30:200:0.5cm) [halfarrow2];
\draw [black] (5.5,0) .. controls (6.25,-0.75) and (8.6,2.5) .. (7.757,2.757) [halfarrow1];

\draw [red]  (5-0.353,0.353) .. controls (2.9,-1.1) and (9.1,-1.6) .. (7.757,-2.757) [halfarrow2];
\draw [gray]  (5-0.353,0.353) .. controls (5.8,1.3) and (6.3,-3.7) .. (7.757,-2.757) [halfarrow1];

\draw (10,0) circle (0.8pt);
\draw (12.545,-2.545) circle (0.8pt); 
\draw (12.545,2.545) circle (0.8pt);
\draw [densely dashed] (10,0) -- (10-0.848,0.848);   
\draw [densely dashed] (10,0) -- (10-0.848,-0.848);
\draw [densely dashed] (12.545,-2.545) -- (13.39,-3.39);
\draw [densely dashed] (12.545,2.545) -- (13.39,3.39);

\draw [black] (10-0.353,-0.353) .. controls (10.8,-1.3) and (11.3,3.7) .. (12.757,2.757) [halfarrow2]; 
\draw [green]  (10-0.353,0.353) .. controls (10.8,1.3) and (11.3,-3.7) .. (12.757,-2.757) [halfarrow1];
\draw [gray]  (10-0.353,0.353) .. controls (7.9,-1.1) and (14.1,-1.6) .. (12.757,-2.757) [halfarrow2];

\draw (15,0) circle (0.8pt);
\draw (17.545,-2.545) circle (0.8pt); 
\draw (17.545,2.545) circle (0.8pt);
\draw [densely dashed] (15,0) -- (15-0.848,0.848);   
\draw [densely dashed] (15,0) -- (15-0.848,-0.848);
\draw [densely dashed] (17.545,-2.545) -- (18.39,-3.39);
\draw [densely dashed] (17.545,2.545) -- (18.39,3.39);

\draw [green]  (15-0.353,0.353) .. controls (12.9,-1.1) and (19.1,-1.6) .. (17.757,-2.757) [halfarrow2];

\draw (15.5,0) node {$a$};
\draw (17,2.545) node {$b$};
\draw (17,-2.545) node {$d$};
\end{tikzpicture}
}
      \end{center}
    \caption{Intersections for $a=c$ and $\varphi > 0$.}
    \label{fig:int_a=c}
   \end{figure}
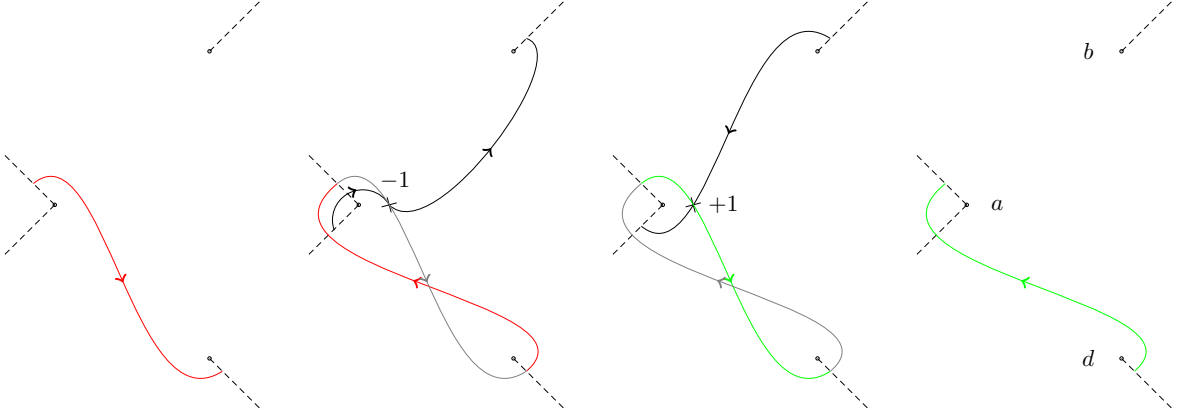
   Lemma \ref{lemma:sxt=sx} (2) gives us $s(\xt) \equiv s_a$, which implies (as claimed for $\varphi > 0$)
    \begin{align*}
    & s_+ = k-s(\xt)+1-k \equiv 1-s_a \mod m, \\
    & s_- = k-s(\xt)-k \equiv -s_a \mod m.
    \end{align*}
  The case $\varphi < 0$ is easily derived by symmetry: if we mirror Figure \ref{m-fig:int_a=c} at the horizontal line through $a$ we are in case (iv).
  There, the intersection is positive if $\cyabk$ and $\cyadl$ start on the same sheet and negative if $\cyadl$ starts one sheet below $\cyabk$.

\begin{lemma}\label{lemma:sxt=sx}
  With the choices made in the proof of Theorem \ref{m-thm:intsec_numb} the following statements hold
  \begin{itemize}
   \item[(1)] $s(\xt) \equiv s_b \bmod m$ in case {\upshape{(ii)}},
    \item[(2)]  $s(\xt) \equiv s_a \bmod m$ in the cases {\upshape{(iii)}} and {\upshape{(iv)}}.
  \end{itemize}
\end{lemma}
 \begin{proof}
 Starting from equation \eqref{m-eq:sxt}, for all $x \in \C \setminus X$ we have
   \begin{align*}
   s(x) & = \frac{m}{2\pi} \arg \left( \frac{\ycd(x)}{\yab(x)} \right)
   \equiv \frac{m}{2\pi} \left( \arg \left( \frac{(1- \ucd(x)^2)\mr}{(1-\uab(x)^2)\mr} \right) + \arg \left( \frac{ \ccd \ytcd(x)}{\cab
   \ytab(x)} \right) \right) \\
   & \equiv \frac{1}{2\pi} \left( \arg(1+\ucd(x)) + \arg(1-\ucd(x)) - \arg(1+\uab(x)) - \arg(1-\uab(x)) \right) \\
   & +  \frac{m}{2\pi} \left( \arg \left( \frac{ \ccd \ytcd(x)}{\cab\ytab(x)} \right)\right) \mod m.
  \end{align*}
 In case (ii) we have $b = c$ and denote $\varphi_0 = \arg \left( \frac{b-a}{d-c} \right)$. Then, we can parametrize all points $\xt \ne b$ on
 the upper bisectrix of $[a,b]$ and
 $[a,d]$ (see Figure \ref{m-fig:set_v_both}) via
 \begin{align*}
  \xt & = \xbd(-1+t \exp(i(\pi+\varphi_0)/2)) \text{  as well as} \\
  \xt & = \xab(1-t \exp(-i (\pi+\varphi_0)/2))
 \end{align*}
 for some $t > 0$. Therefore,
 \begin{align*}
   & \arg (1+\ubd(\xt)) = \frac{\pi+\varphi_0}{2} \text{ and} \\
   & \arg (1-\uab(\xt)) = -\frac{\pi+\varphi_0}{2}.
 \end{align*}
 For $\xt$ chosen close enough to $b$ we have that $[\xt,b[ \subset V$ and the shifting function $s(\xt)$ is constant as $\xt$ tends towards $b$. Hence,
 we can compute its value at $\xt$ as
 \begin{align*}
  s(\xt) & \equiv \frac{1}{2\pi} \left( \pi + \varphi_0 + \arg(1-\ubd(\xt)) - \arg(1+\uab(\xt)) +  m \arg \left( \frac{ \cbd \ytcd(\xt)}{\cab\ytab(\xt)} \right)\right) \\
	 & \equiv \frac{1}{2\pi} \left ( \varphi +\arg(1-\ubd(b)) - \arg(1+\uab(b)) +  m \arg \left( \frac{ \cbd \ytbd(b)}{\cab\ytab(b)} \right)\right) \\
	 & \equiv \frac{1}{2\pi} \left( \varphi +\arg(2) - \arg(2) +  m \arg \left( \frac{ \ccd \ytbd(b)}{\cab\ytab(b)} \right)\right) \equiv s_b \mod m,
 \end{align*}
thus proving (1).

  In the cases (iii) and (iv) we have $a = c$ and denote $\varphi =  \arg \left( \frac{b-a}{d-c} \right)$. For $\varphi > 0$ we can parametrize all points $\xt \ne a$ on the inner
  bisectrix of $[a,b]$ and $[a,d]$ (see Figure \ref{m-fig:set_v_both}) via
 \begin{align*}
  \xt & = \xad(-1+t \exp(i \varphi /2)) \text{  as well as} \\
  \xt & = \xab(-1+t \exp(-i \varphi/2))
 \end{align*}
 for some $t > 0$. Therefore,
 \begin{align*}
   & \arg (1+\uad(\xt)) = \frac{\varphi}{2} \text{ and} \\
   & \arg (1+\uab(\xt)) = -\frac{\varphi}{2}.
 \end{align*}
 As before, we let $\xt$ tend towards $a$ and compute the shifting function at $\xt$ as
 \begin{align*}
  s(\xt) & \equiv \frac{1}{2\pi} \left( \varphi + \arg(1-\uad(\xt)) - \arg(1+\uab(\xt)) +  m \arg \left( \frac{ \cad \ytad(\xt)}{\cab\ytab(\xt)} \right)\right) \\
	 & \equiv \frac{1}{2\pi} \left ( \varphi +\arg(1-\uad(a)) - \arg(1-\uab(a)) +  m \arg \left( \frac{ \cad \ytad(a)}{\cab\ytab(a)} \right)\right) \\
	 & \equiv \frac{1}{2\pi} \left( \varphi +\arg(2) - \arg(2) +  m \arg \left( \frac{ \ccd \ytad(a)}{\cab\ytab(a)} \right)\right) \equiv s_a \mod m.
 \end{align*}
 The case $\varphi < 0$ is proved analogously.
\end{proof}

\begin{rmk}
  The intersection numbers given by Theorem \ref{m-thm:intsec_numb} are independent of the choices of $\xt$ that were made in the proof. This approach works for any $\xt \in V$.

  Even though the value of
  $s(\xt)$ changes, if we choose $\xt$ in a different connected component of $V$, e.g.\ on the lower bisectrix in case (ii),
  the parametrization of the bisectrix and the corresponding arguments will change accordingly.
\end{rmk}

\section{Numerical integration}\label{sec:numerical_integration}

As explained in Section \ref{m-subsec:numerical_integration}, the periods
of the generating cycles $\gamma\in \Gamma$ are expressed in terms of
elementary integrals
\eqref{m-eq:elem_num_int}
\begin{equation*}
    I_{a,b}(i,j) = \int_{-1}^1\frac{u^{i-1}\du}{(1-u^2)^{\frac jm}\ytab(u)^j}
\end{equation*}
where $(a,b)\in E$ and $\omega_{i,j}\in\W$.
We restrict the numerical analysis to this case.

In this section, we denote by $α$ the value $1-j/m$, which is the crucial parameter for numerical integration.
Note that $α=1/2$ for hyperelliptic curves, while for general superelliptic curves $α$ ranges
from $1/m$ to $\frac{m-1}m$ depending on the differential form $\omega_{i,j}$ considered.

We study here two numerical integration schemes which are suitable for arbitrary
precision computations:
\begin{itemize}
    \item
        the double-exponential change of variables is completely general \cite{Molin2010} and its robustness
allows to compute rigorously all integrals of periods in a very unified setting
even with different values of $\alpha$;
\item in the special case of hyperelliptic curves however,
    the Gauss-Chebychev method \cite[25.4.38]{AbramowitzStegun} applies and
    provides a better scheme (fewer and simpler integration points).
\end{itemize}
For $m > 2$, the periods could also be computed using general Gauss-Jacobi integration
of parameters $\alpha,\alpha$. However, a different scheme has to be computed for each $\alpha$
and it now involves computing roots of general Jacobi polynomials to large accuracy, which
makes it hard to compete with the double-exponential scheme.

\begin{rmk}
    Even for hyperelliptic curves it can happen that the double exponential scheme outperforms Gauss-Chebychev on
    particular integrals. This is easy to detect in practice and we can always switch to
    the best method.
\end{rmk}

\subsection{Double-exponential integration}\label{m-subsec:de_int}

Throughout this section, $λ\in[1,\frac{π}2]$ is a fixed parameter.
By default the value $λ=\frac{π}2$ is a good choice, however
smaller values may improve the constants. We will not address
this issue here.
\medskip

Using the double-exponential change of variable
\begin{equation}
    \label{eq:de_change}
u=\tanh(λ\sinh(t)),
\end{equation}
the singularities of \eqref{eq:elem_num_int} at $\pm1$ are pushed to infinity and
the integral becomes
\begin{equation*}
    I_{a,b}(i,j) = \int_\R g(t)\dt
\end{equation*}
with
\begin{equation*}
   g(t) = \frac{u(t)^{i-1}}{\ytab(u(t))^j}\frac{λ\cosh(t)}{\cosh(λ\sinh(t))^{2α}}.
\end{equation*}

Let
\begin{equation*}
    Z_r = \set{\tanh(λ\sinh(z)), -r<\Im(z)<r }
\end{equation*}
be the image of the strip of width $2r$ under the change of
variable \eqref{m-eq:de_change}.

Since we can compute the distance of each branch point $u_i$ to
both $[-1,1]$ and its neighborhood $Z_r$ (see \S \ref{m-subsec:de_case}), we obtain
  \begin{lemma}
      There exist explicitly computable
      constants $M_1$, $M_2$ such
      that
      \begin{itemize}
          \item $\abs{\frac{u^{i-1}}{\ytab(u)^{j}}}\leq M_1$ for all $u\in[-1,1]$,
          \item $\abs{\frac{u^{i-1}}{\ytab(u)^{j}}}\leq M_2$ for all $u\in Z_r$.
      \end{itemize}
  \end{lemma}

We also introduce the following quantities
\begin{equation*}
    \begin{cases}
    X_r &=\cos(r)\sqrt{\frac{π}{2λ\sin r}-1} \\[0.2cm]
    B(r,α) &=
    \frac{2}{\cos r}
    \left(
        \frac{X_r}2\left(\frac1{\cos(λ\sin r)^{2α}}+\frac1{X_r^{2α}}\right)
        +\frac{1}{2α\sinh(X_r)^{2α}}
    \right).
    \end{cases}
\end{equation*}

Once we have computed the two bounds $M_1$, $M_2$ and the constant $B(r,α)$,
we obtain a rigorous integration scheme as follows:
\begin{thm}
    \label{thm:de_int}
    With notation as above, for all $D>0$, choose $h$ and $N$ such that
    \begin{equation}
    \label{eq:de_parameters}
        \begin{cases}
            h \le \frac{2πr}{D + \log(2M_2 B(r,α) + e^{-D})}\\[0.2cm]
            Nh \ge \asinh\left(\frac{D+\log(\frac{2^{2α+1}M_1}{α})}{2αλ}\right),
        \end{cases}
    \end{equation}
    then
    \begin{equation*}
        \abs{
            I_{a,b}(i,j)
            - h\sum_{k=-N}^N
            w_k \frac{u_k^{i-1}}{\ytab(u_k)^j}
        } \leq e^{-D},
    \end{equation*}
    where
    \begin{equation*}
        \begin{cases}
            u_k = \tanh(λ\sinh(kh)),\\[0.2cm]
            w_k = \frac{λ\cosh(kh)}{\cosh(λ\sinh(kh))^{2α}}.
        \end{cases}
    \end{equation*}
\end{thm}

The proof follows the same lines as the one in \cite[Thm. 2.10]{Molin2010}:
we write the Poisson formula on $h\Z$ for the function $g$
\begin{equation*}
    \underbrace{h\sum_{\abs{k}>N}g(kh)}_{e_T}
 + h\sum_{k=-N}^N g(kh)
 = \int_\R g
 + \underbrace{\sum_{k\in\Z^\ast} \hat g\left(\frac{k}{h}\right)}_{e_Q}
\end{equation*}
and control both error terms $e_T$ and $e_Q$ by Lemma \ref{m-lem:de_error_trunc}
and \ref{m-lem:de_error_quad} below. The actual parameters $h$ and $N$ follow
by bounding each error by $e^{-D}/2$.

\begin{lemma}[truncation error]
    \label{lem:de_error_trunc}
    \begin{equation*}
        \sum_{\abs{k}>N}\abs{hg(kh)}
        \leq \frac{2^{2α} M_1}{αλ}\exp(-2αλ\sinh(nh)).
    \end{equation*}
\end{lemma}
\begin{proof}
    We bound the sum by the integral of a decreasing function
    \begin{align*}
        \sum_{\abs{k}>N}\abs{hg(kh)}
        &\leq2M_1\int_{Nh}^\infty\frac{λ\cosh(t)}{\cosh(λ\sinh(t))^{2α}}
        =2M_1\int_{λ\sinh(Nh)}^\infty\frac{\dt}{\cosh(t)^{2α}}\\
        &\leq 2^{2α+1} M_1\int_{λ\sinh(Nh)}^\infty e^{-2αt}\dt
        = \frac{2^{2α} M_1}{α}e^{-2αλ\sinh(Nh)}.
    \end{align*}
\end{proof}

\begin{lemma}[discretization error]
    \label{lem:de_error_quad}
    With the current notations,
    \begin{equation*}
        \sum_{k\neq0}\abs{\hat g\left(\frac kh\right)}
        \leq
        \frac{M_2B(r,α)}{e^{2πr/h}-1}.
    \end{equation*}
\end{lemma}
\begin{proof}
We first bound the Fourier transform by a shift of contour
\begin{equation*}
    \forall X>0, \hat g(\pm X) = e^{-2πXr} \int_{\R} g(t\mp ir) e^{-2iπtX}\dt
\end{equation*}
so that
\begin{equation*}
    \sum_k \abs{\hat g\left(\frac kh\right)}
    \leq
    \frac{2M_2}{e^{2πr/h}-1}\int_\R \abs{
    \frac{λ\cosh(t+ir)}{\cosh(λ\sinh(t+ir))^{2α}}}\dt.
\end{equation*}

Now the point $λ\sinh(t+ir) = X(t)+iY(t)$ lies on the hyperbola
$Y^2 =λ^2(\sin^2r+\tan^2 rX^2)$, and
\begin{equation*}
    \begin{cases}
    \abs{λ\cosh(t+ir)} &\leq λ\cosh(t) =\frac{X'(t)}{\cos(r)}\\[0.2cm]
    \abs{\cosh(X+iY)}^2 &= \sinh(X)^2+\cos(Y)^2,
    \end{cases}
\end{equation*}
so that
\begin{equation*}
    \int_\R \abs{
    \frac{λ\cosh(t+ir)}{\cosh(λ\sinh(t+ir))^{2α}}}\dt
    \leq
    \frac{2}{\cos r}\int_0^\infty\frac{\d X}{(\sinh(X)^2+\cos(Y)^2)^α}.
\end{equation*}
For $X_0=0$ we get $Y_0=λ\sin r<\frac{π}2$, and $Y_r=\frac{π}2$ for
$X_r=\cos(r)\sqrt{\frac{π}{2Y_0}-1}$.

  We cut the integral at $X=X_r$ and write
  \begin{align*}
      \int_0^{X_r}\frac{\d X}{(\sinh(X)^2+\cos(Y)^2)^α}
      & \leq \int_0^{X_r}\frac{\d X}{(X^2+\cos^2Y)^α} \\
      \int_{X_r}^\infty\frac{\d X}{(\sinh(X)^2+\cos(Y)^2)^α}
      & \leq \int_{X_r}^\infty\frac{\d X}{(\sinh X)^{2α}}.
  \end{align*}

  We bound the first integral by convexity:
  since $Y(X)$ is convex and $\cos$ is concave decreasing for $Y\leq Y_r$ we
  obtain by concavity of the composition
  \begin{equation*}
      \forall X\leq X_r, \cos(Y)\geq \cos(Y_0)\left(1-\frac{X}{X_r}\right).
  \end{equation*}
  Now $X^2+\cos^2Y\geq P_2(X)$ where
  \begin{equation*}
     P_2(X) = \left(1+\frac{\cos^2(Y_0)}{X_r^2}\right)X^2 -2\frac{\cos^2(Y_0)}{X_r}X+\cos^2(Y_0)
  \end{equation*}
  is a convex quadratic, so $X\mapsto P_2(X)^{-α}$ is still convex and the integral
  is bounded by a trapezoid
  \begin{equation*}
  \int_0^{X_r}\frac{\d X}{P_2(X)^α}\leq \frac{X_r}2\left(P_2(0)^{-α}+P_2(X_r)^{-α}\right)
      = \frac{X_r}2\left(\frac1{\cos(Y_0)^{2α}}+\frac1{X_r^{2α}}\right).
  \end{equation*}

  For the second integral we use
  $\sinh(X)\geq\sinh(X_r)e^{X-X_r}$ to obtain
  \begin{equation*}
      \int_{X_r}^\infty \frac{\d X}{\sinh(X)^{2α}} \leq \frac1{2α\sinh(X_r)^{2α}}.
  \end{equation*}
\end{proof}

\subsection{Gauss-Chebychev integration}
\label{m-subsec:gauss_chebychev_integration}

In the case of hyperelliptic curves, we have $α=\frac12$ (and $j=1$) and the integral
\begin{equation*}
    \int_{-1}^1\frac{\varphi_{i,1}(u)}{\sqrt{1-u^2}}\du
\end{equation*}
can be efficiently handled by Gaussian integration with weight
$1/\sqrt{1-u^2}$,
for which the corresponding orthogonal polynomials are
Chebychev polynomials.

In this case, the integration formula is particularly
simple: there is no need to actually compute the Chebychev polynomials
since their roots are explicitly given as cosine functions \cite[25.4.38]{AbramowitzStegun}.
\begin{thm}[Gauss-Chebychev integration]
    Let $g$ be holomorphic around $[-1,1]$. Then for all
    $N$, there exists $\xi \in ]-1,1[$ such that
    \begin{equation}
        \label{eq:gauss_chebychev}
        \int_{-1}^1\frac{g(u)}{\sqrt{1-u^2}}\du
        - \sum_{k=1}^N w_k g(u_k)
        = \frac{π2^{2N+1}}{2^{4N}}\frac{g^{(2N)}(\xi)}{(2N)!}
     = E(N),
    \end{equation}
    with constant weights $w_k = w =\frac{π}N$
    and nodes $u_k = \cos\left(\frac{2k-1}{2N}π\right)$.
\end{thm}

Moreover, very nice estimates on the error $E(N)$ can by obtained by applying
the residue theorem on an ellipse $ε_r$ of the form
\begin{equation*}
    ε_r = \set{z, \abs{z-1}+\abs{z+1} = 2\cosh(r) }.
\end{equation*}

  \begin{figure}[H] \begin{center}
      \includegraphics[width=5cm,page=1]{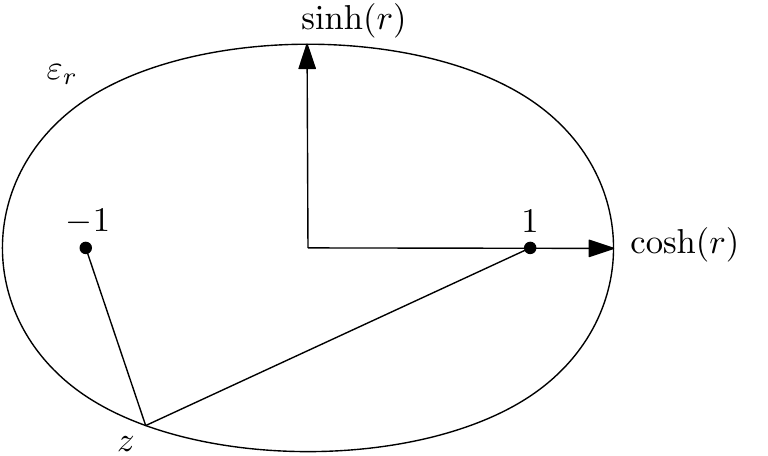}
  \end{center} \caption{ellipse parameters.}
  \label{fig:ellipse1} \end{figure}

\begin{thm}[\cite{ChawlaJain68},Theorem 5]
    Let $r>0$ such that $g$ is holomorphic on $ε_r$. Then
    the error in \eqref{m-eq:gauss_chebychev} satisfies
    \begin{equation*}
        \abs{E(N)}\leq \frac{2πM(r)}{e^{2rN}-1}
    \end{equation*}
    where $M(r)=\max\set{\abs{f(z)},z\in ε_r}$.
\end{thm}

Now we use this theorem with a function
$g_{i}(u)=\frac{u^i}{\sqrt{Q(u)}}$ for an explicitly factored
polynomial $Q(u)=\prod(u-u_k)$, so that the error can be
explicitly controlled.

\begin{lemma}
    \label{lem:param_r}
    Let $r>0$ be such that $2\cosh(r)<\abs{u_k-1}+\abs{u_k+1}$ for all
    roots $u_k$ of $Q$,
    then there exists an explicitly computable
    constant $M(r)$ such that for all $u\in ε_r$
    \begin{equation*}
        \abs{\frac{u^{i-1}}{\ytab(u)}}\leq M(r).
    \end{equation*}
\end{lemma}
\begin{proof}
We simply compute the distance
        $d_r(u_k) = \inf_{z\in ε_r}\abs{z-u_k}$
 from a root $u_k$ to the ellipse $ε_r$, and let
 $M(r) =  \frac{\cosh(r)^{i-1}}{\sqrt{\prod d_r(u_k)} }$.
 For simplicity, we can use the triangle inequality
 $d_r(u_k)\geq \cosh(r_k)-\cosh(r)$, where
 $2\cosh(r_k)=\abs{u_k-1}+\abs{u_k+1}$.
\end{proof}

\begin{thm}
    \label{thm:gc_int}
    With $r$ and $M(r)$ satisfying Lemma \ref{m-lem:param_r},
    for all $N$ such that
    \begin{equation*}
        \label{eq:Ngc}
        N \geq \frac{D+\log(2πM(r))+1}{2r},
    \end{equation*}
    we have
    \begin{equation*}
        \abs{I_{a,b}(i,1)
        - \frac{π}N\sum_{k=1}^N\frac{u_k^{i-1}}{\ytab(u_k)}}
            \leq e^{-D},
    \end{equation*}
    where $u_k=\cos\left(\frac{2k-1}{2N}π\right)$.
\end{thm}

More details on the choice of $r$ and the computation of $M(r)$
are given in \S \ref{m-par:gc_int_r}.

  \section{Computing the Abel-Jacobi map}\label{sec:comp_ajm}

   Here we are concerned with explicitly computing the Abel-Jacobi map of degree zero divisors; for a general introduction see Section \ref{m-sec:ajm}.

   Assume for this section that we have already computed a big period period matrix (and all related data) following the Strategy from Section \ref{m-sec:strat_pm}.

   Let $D = \sum_{P \in \cu} v_P P \in \Div^0(\cu)$. After choosing a basepoint $P_0 \in \cu$, the computation of $\AJ$ reduces (using linearity) to
   \begin{equation*}
     \AJ([D]) \equiv \sum_{P \in \cu} v_P \int_{P_0}^P \bar\w \mod \Lambda.
   \end{equation*}
  For every $P \in \cu$, $\int_{P_0}^P \bar\w$ is a linear combination of vector integrals of the form
  \begin{align*}
    \int_{P_0}^{P_k} \bar\w\quad \text{(see \S \ref{m-subsec:ajm_ram_pts}),} \quad
    \int_{P_k}^{P} \bar\w \quad \text{(see \S \ref{m-subsec:ajm_finite})}
    \quad \text{and} \quad \int_{P_0}^{P_{\infty}} \bar\w\quad \text{(see \S \ref{m-subsec:ajm_infty}),} \quad \text{where}
  \end{align*}
  \begin{itemize}
   \item $\bar\w$ is the vector of differentials in $\W$,
   \item $P = (x_P,y_P) \in \cu$ is a finite point on the curve,
   \item $P_k = (x_k,0) \in \cu$ is a finite ramification point, i.e.\ $x_k \in X$, and
   \item $P_{\infty} \in \cu$ is an infinite point.
  \end{itemize}

   Typically, we choose as basepoint the ramification point $P_0 = (x_0,0)$, where $x_0 \in X$ is the root of the spanning tree $G = (X,E)$.

  Finally, the resulting vector integral has to be reduced modulo the period lattice $\Lambda$, which is covered in \S \ref{m-subsec:lat_red}.

   \begin{rmk}[Image of Abel-Jacobi map] \
  For practical reasons, we will compute the
  image of the Abel-Jacobi map in the canonical torus $\R^{2g} / \Z^{2g}$.
This representation has the following advantages:
    \begin{itemize}
     \item[$\bullet$] Operations on the Jacobian variety $\Jac(\cu)$ correspond to operations in $\R^{2g}/\Z^{2g}$.
     \item[$\bullet$] $m$-torsion divisors under $\AJ$ are mapped to vectors of rational numbers with
     denominator dividing $m$.
         \end{itemize}
   \end{rmk}

  \subsection{Between ramification points}\label{subsec:ajm_ram_pts}

  Suppose we want to integrate $\bar\w$ from $P_0 =(x_0,0)$ to $P_k = (x_k,0)$. By construction there exists a path
 $(x_0=x_{k_0},x_{k_1},\dots,x_{k_{n-1}},x_{k_t}=x_k)$ in the spanning tree which connects $x_0$ and $x_k$. Thus, the integral splits into
  \begin{align*}
    \int_{P_0}^{P_k} \bar\w = \sum_{j = 0}^{t-1}  \int_{P_{k_j}}^{P_{k_{j+1}}} \bar\w.
  \end{align*}
  Denote $a = x_{k_j}, b = x_{k_{j+1}} \in X$. From \S \ref{m-subsec:cycles_homo} we know that for $(a,b) \in E$ a smooth path between $P_a=(a,0)$ and $P_b=(b,0)$ is given by
  \begin{align*}
   \gamma_{[a,b]}^{(0)} = \{  (x,\yab(x))  \mid  x \in [a,b]  \}.
  \end{align*}
  Let $\w_{i,j} \in \W$ be a differential. According to the proof of Theorem \ref{m-thm:periods} the corresponding integral is given by
  \begin{align*}
   \int_{\gamma_{[a,b]}^{(0)}} \w_{i,j}  & =
   C_{a,b}^{-j} \left(\frac{b-a}{2}\right)^i \int_{-1}^1 \frac{\varphi_{i,j}(u)}{(1-u^2)^{\frac{j}{m}}}  \du,
  \end{align*}
  which is (up to the constants) an elementary integral \eqref{m-eq:elem_ints}
  and has already been evaluated during the period matrix computation.

  \begin{rmk}
  Moreover, the image of the Abel-Jacobi map between ramification points is $m$-torsion, i.e. for any two $k,j \in \{1,\dots,n\}$ we have
  \begin{equation}\label{eq:m_tors}
    m \int_{P_j}^{P_k} \bar\w \equiv \AJ( [ mP_k - mP_j ] ) \equiv 0 \mod  \Lambda,
  \end{equation}
  since $\div\left( \frac{x-x_k}{x-x_j} \right) = mP_k - mP_j$ is a principal divisor.
  \end{rmk}

  \subsection{Reaching non-ramification points}\label{subsec:ajm_finite}

  Let $P = (x_P,y_P) \in \cu$ be a finite point and $P_a = (a,0)$ a ramification point such that $X\cap\,]a,x_P]\,=\varnothing$. In order to define a smooth path between $P$ and $P_a$
  we need to find a suitable analytic branch of $\cu$.

  This can be done following the approach in \S \ref{m-subsubsec:analytic_branches}, the only difference being that
  $x_P$ is not a branch point. Therefore, we are going to adjust the definitions and highlight the differences.

  Let  $\uaxp$ be the affine linear transformation
  that maps $[a,x_P]$ to $[-1,1]$.
  Similar to \eqref{m-eq:uab_image} we split up the image of $X$ under $\uaxp$ into subsets,
  but this time
 \begin{equation*}
  \uaxp(X) = \{-1\} \cup U^+ \cup U^-.
 \end{equation*}
  Then, $\ytaxp(u)$ can be defined exactly as in $\eqref{m-eq:ytab}$ and
  is holomorphic in a neighbourhood $\epsilon_{a,x_P}$ of $[-1,1]$.
  The term corresponding to $a$, that is
  \begin{equation*}
   \sqrt[m]{1+u},
  \end{equation*}
   has a branch cut $]-\infty,-1]$ and is holomorphic on the complement of this cut.

  Now we can define a branch of the curve,
   that is analytic in a neighbourhood $V_{a,x_P}$ of $]a,x_P]$, by
  \begin{equation*}
    \yaxp(x) =   C_{a,x_P} \ytaxp( \uaxp(x) ) \sqrt[m]{1 + \uaxp(x)},
  \end{equation*}
  where
   \begin{equation*}
      C_{a,x_P} = \left(\frac{x_P-a}{2}\right)^{\frac{n}{m}} e^{\frac{\pi i}{m}(\#U^+ \bmod 2)},
  \end{equation*}
  so that the statements of Proposition \ref{m-prop:yab} continue to hold for $\ytaxp$ and $\yaxp$,
  if we choose the sets $\epsilon_{a,x_P}$ and $V_{a,x_P}$ as if $x_P$ was a branch point.

  Therefore, the lifts of $[a,x_P]$ to $\cu$ are given by
    \begin{equation*}\label{eq:def_path_axp}
      \gamma^{(l)}_{[a,x_P]} = \{  (x,\zeta^l \yaxp(x))  \mid  x \in [a,x_P]  \}, \quad l \in \Z/m\Z.
   \end{equation*}
  In order to reach $P = (x_P,y_P)$ we have to pick the correct lift. This is done by computing a \emph{shifting number} $s \in \Z/m\Z$ at the endpoint $x_P$:
  \begin{equation*}
      \label{eq:saj}
   \zeta^s = \frac{y_P}{\yaxp(x_P)} = \frac{y_P}{ C_{a,x_P} \ytaxp( \uaxp(x_P) ) \sqrt[m]{2}  }
  \end{equation*}
  Consequently, $\gamma^{(s)}_{[a,x_P]}$ is a smooth path between $P_a$ and $P$ on $\cu$. We can now state the main theorem of this section.

  \begin{thm}\label{thm:ajm_finite_int}
  Let $\w_{i,j} \in \WM$ be a differential. With the choices and notation as above we have
 \begin{equation*}\label{eq:aj_int}
       \int_{P_a}^P \w_{i,j} = \zeta^{-sj} C_{a,x_P}^{-j} \left(\frac{x_P-a}{2}\right)^i \int_{-1}^1 \frac{\varphi_{i,j}(u)}{(1+u)^{\frac{j}{m}}}  \du,
  \end{equation*}
   where
   \begin{equation*}
    \varphi_{i,j}  = \left(u+\frac{x_P+a}{x_P-a}\right)^{i-1} \ytaxp(u)^{-j}
   \end{equation*}
   is holomorphic in a neighbourhood $\epsilon_{a,x_P}$ of $[-1,1]$
   and
   \begin{equation*}
   s = \frac{m}{2\pi} \arg\left(  \frac{y_P}{C_{a,x_P} \ytaxp( \uaxp(x_P) )} \right).
   \end{equation*}
  \end{thm}
  \begin{proof}
    We have
    \begin{align*}
     \int_{P_a}^P \w_{i,j}  & =  \int_{\gamma^{(s)}_{[a,x_P]}} \frac{x^{i-1}}{y^j}  \dx  =  \zeta^{-sj} \int_a^{x_P} \frac{x^{i-1}}{\yaxp(x)^j}  \dx \\  & =
     \zeta^{-sj} C_{a,x_P}^{-j}  \int_a^{x_P} \frac{x^{i-1}}{(1+\uaxp(x))^{\frac{j}{m}}\ytaxp(\uaxp(x))^j}  \dx
  \end{align*}
   Applying the transformation $u = \uaxp(x)$ introduces the derivative $\dx = \left(\frac{x_P-a}{2}\right) \du$. Hence
  \begin{align*}
    \int_{P_a}^P \w_{i,j} & =  \zeta^{-sj} C_{a,x_P}^{-j} \left(\frac{x_P-a}{2}\right) \int_a^{x_P} \frac{\xaxp(u)^{i-1}}{(1+u)^{\frac{j}{m}}\ytaxp(u)^j}  \du \\ & =
   \zeta^{-sj} C_{a,x_P}^{-j} \left(\frac{x_P-a}{2}\right)^i \int_a^{x_P} \frac{\left(u + \frac{x_P+a}{x_P-a}\right)^{i-1}}{(1+u)^{\frac{j}{m}}\ytaxp(u)^j}  \du.
  \end{align*}
  The statement about holomorphicity of $\varphi_{i,j}$ is implied, since
  Proposition \ref{m-prop:yab} holds for $\ytaxp$ and $\yaxp$ as discussed
  above.
  \end{proof}

  \begin{rmk}\label{rmk:ajm_finite_int}
   By Theorem \ref{m-thm:ajm_finite_int}, the problem of integrating $\bar\w$ from $P_0$ to $P$ reduces to numerical integration of
    \begin{equation*}
       \int_{-1}^1 \frac{\varphi_{i,j}(u)}{(1+u)^{\frac{j}{m}}}  \du.
   \end{equation*}
   Although these integrals are singular at only one end-point, they
   can still  be computed using the double-exponential estimates presented in Section
   \ref{m-sec:numerical_integration} (this is not true for the Gauss-Chebychev method).
   \end{rmk}

  \subsection{Infinite points}\label{subsec:ajm_infty}

  Recall from \S \ref{m-subsec:se_def} that there are $\delta = \gcd(m,n)$ points $P_{\infty}^{(i)}$ at infinity on our projective curve $\cu$, so
  we introduce the set $\mathcal{P}
  = \{ P_{\infty}^{(1)},\dots, P_{\infty}^{(\delta)} \}$.

  Suppose we want to integrate from $P_0$ to $P_{\infty} \in \mathcal{P}$, which is equivalent to computing the Abel-Jacobi map of the divisor
  $D_{\infty} = P_{\infty} - P_0$.

  Our strategy is to explicitly apply Chow's moving lemma to $D_{\infty}$: we construct a principal divisor $D \in \Prin(\cu)$ such that $\supp(D) \cap \mathcal{P} = \{ P_{\infty} \}$
  and $\ord_{P_{\infty}}(D) = \pm 1$. Then, by definition of the Abel-Jacobi map,
  \begin{equation*}
  \AJ([D_{\infty} \mp D]) \equiv \AJ([D_{\infty}]) \equiv \int_{P_0}^{P_{\infty}} \bar\w \mod \Lambda
  \end{equation*}
  and $\supp(D_{\infty} \mp D) \cap \mathcal{P} = \varnothing$.

  The exposition in this paragraph will explain the construction of $D$, while distinguishing three different cases.

  In the following denote by  $-\mu,\nu > 0$ the coefficients of the Bézout identity
  \begin{equation*}
    \mu m + \nu n = \delta.
  \end{equation*}

  \begin{rmk}
   Note that there are other ways of computing $\AJ([D_{\infty}])$. For instance, using transformations or direct numerical
   integration. Especially in the case $\delta = m$ a transformation (see Remark \ref{m-rmk:moebius}) is the better option and may be used in practice.
   The advantage of this approach is that we can stay in our setup, i.e. we can compute solely on $\caff$ and
   keep the integration scheme.
  \end{rmk}

  \subsubsection{Coprime degrees}\label{subsec:ajm_inf_cop}

  For $\delta = 1$ there is only one infinite point $\mathcal{P} = \{ P_{\infty} \}$ and
  we can easily compute  $\AJ( [D_{\infty}] )$ by adding a suitable
  principal divisor $D$
    \begin{align*}
      &\div(y^{\nu})=  \nu \sum_{k = 1}^n P_k - \nu n P_{\infty},\\
      &\div((x-x_0)^{\mu})  =  \mu m P_0 - \mu m P_{\infty} ,\\
      D  =  & \div(y^{\nu}(x-x_0)^{\mu})  = \nu \sum_{k = 1}^n P_k + \mu m P_0 - P_{\infty}.
    \end{align*}
    We immediately obtain
    \begin{align*}
     \AJ( [D_{\infty}] )  & \equiv  \AJ([D_{\infty} + D])  =  \AJ( [ \nu \sum_{k = 1}^n P_k + (\mu m - 1) P_0 ]  ) \\ &
      \equiv   \nu \sum_{k=1}^n \int_{P_0}^{P_k} \bar\w \mod \Lambda
    \end{align*}
    and conclude that $\AJ( [D_{\infty}] )$ can be expressed in terms of integrals between
    ramification points (see \S \ref{m-subsec:ajm_ram_pts}).
  \begin{rmk}\label{rmk:sum_infty_pts}
  In general, the principal divisor
  \begin{equation*}
    D  := \div(y^{\nu}(x-x_0)^{\mu}) = \nu \sum_{k = 1}^n P_k + \mu m P_0 - \sum_{l = 1}^{\delta} P^{(l)}_{\infty}
  \end{equation*}
  yields the useful relation
  \begin{equation*}
  \nu \sum_{k=1}^n \int_{P_0}^{P_k} \bar\w \equiv \sum_{l = 1}^{\delta} \int_{P_0}^{P^{(l)}_{\infty}} \bar\w \mod \Lambda.
  \end{equation*}
  \end{rmk}

 \subsubsection{Non-coprime degrees}\label{subsec:ajm_inf_ncop}

  For $\delta > 1$ the problem becomes a lot harder. First we need a way to distinguish between the infinite points in $\mathcal{P}
  = \{ P_{\infty}^{(1)},\dots, P_{\infty}^{(\delta)} \}$ and second they are singular points
  on the projective closure of our affine model $\caff$
  whenever $m \ne \{n,n\pm1\}$.

  As shown in \cite[\S 1]{CT1996} we obtain a second affine patch of $\cu$ that is non-singular along $\mathcal{P}$
  in the following way:

  Denoting $M = \frac{m}\delta$ and $N = \frac{n}\delta$, we consider the birational transformation
  \begin{equation*}
   (x,y) = \Phi(r,t) = \left(\frac{1}{r^{\nu}t^M},\frac{r^{\mu}}{t^N}\right)
  \end{equation*}
  which results in an affine model
  \begin{equation*}
   \cafft : r^{\delta} = \prod_{k=1}^n (1-x_kr^{\nu}t^M).
  \end{equation*}
  The inverse transformation is given by
  \begin{equation*}
   (r,t) = \Phi^{-1}(x,y) = \left(\frac{y^M}{x^N},\frac{1}{x^{\mu}y^{\nu}}\right).
  \end{equation*}
  Under this transformation the infinite points in $\mathcal{P}$ are mapped to finite points that have the coordinates
  \begin{equation*}
   (r,t) = (\zeta_{\delta}^s,0) \quad s= 1,\dots,\delta,
  \end{equation*}
  where $\zeta_{\delta} = e^{\frac{2\pi i }{\delta}}$.
  Hence, we can describe the points in $\mathcal{P} \subset \cu$ via
   \begin{equation*}
      P_{\infty}^{(s)} = \Phi^{-1}(\zeta_{\delta}^s,0).
   \end{equation*}

   Suppose we want to compute the Abel-Jacobi map of $D_{\infty}^{(s)} = P_{\infty}^{(s)} - P_0$ for $s \in \{1,\dots,\delta\}$.
   Again following our strategy,
   this time working on $\cafft$, we look at the intersection of the vertical line through $(\zeta_{\delta}^s,0)$ with
   $\cafft$
   \begin{equation*}
      E_1 = \div(r - \zeta_{\delta}^s) = \sum_{i = 1}^{d} \left(\zeta_{\delta}^s,t_i^{(s)}\right) - N E_1'
   \end{equation*}
      where the $t_i^{(s)}$ are the zeros of $h(t) = \prod_{k=1}^n (1-x_k\zeta_{\delta}^{s\nu}t^M) - 1 \in \C[t]$ and
    \begin{equation}\label{eq:zero_bp1}
       E_1' = \begin{cases}
	     (m-M) \Phi^{-1}(0,0), \quad \text{if} \; 0 \in X, \\
             \sum_{Q \in\text{pr}_x^{-1}(0)} \Phi^{-1}(Q) \quad \text{otherwise}
            \end{cases}
    \end{equation}
    Note that $E_1$ satisfies $\supp(E_1) \cap \Phi(\mathcal{P}) = \{ (\zeta_{\delta}^s,0) \}$.
    Now, we can define the corresponding principal divisor on $\caff$ by
    \begin{equation*}
       D_1 := \div \left( \frac{y^M}{x^N} - \zeta_{\delta}^s \right);
    \end{equation*}
   then $\ord_{P_{\infty}^{(s)}}(D_1) \ge 1$ by construction.

  \begin{thm}\label{thm:ajm_inf_ord1}
    Assume $\ord_{P_{\infty}^{(s)}}(D_1) = 1$ and $0 \not\in X$. Then, for $s = 1,\dots,\delta$, there exist points $Q_1^{(s)},\dots,Q_{d-1}^{(s)} \in \cu \setminus \mathcal{P}$ such that
    \begin{equation}\label{eq:ajm_inf_ord1}
       \AJ( [D_{\infty}^{(s)}] ) \equiv -\sum_{i = 1}^{n-1} \int_{P_0}^{Q_i^{(s)}} \bar\w \mod \Lambda.
    \end{equation}
  \end{thm}
  \begin{proof}
    First note that $\ord_{P_{\infty}^{(s)}}(D_1) = 1$ implies $M = 1$, i.e. $m = \delta$. Together with the assumption
    $0 \not\in X$, this gives us $\deg(h) = n$.
    Moreover, we can assume that $t_n^{(s)} = 0$ and $t_i^{(s)} \ne 0$ for $i=1,\dots,n-1$.
    Therefore,
    \begin{align*}
      \AJ( [D_{\infty}^{(s)}] )  & \equiv  \AJ([D_{\infty}^{(s)} - D_1])
      \equiv  -\AJ \left( \left[\sum_{i = 1}^{d-1} \Phi(\zeta_{\delta}^s,t_i^{(s)}) - N \sum_{Q \in\text{pr}_x^{-1}(0)} Q  \right] \right)
       \mod \Lambda.
    \end{align*}
    Since  $0 \not\in X$ the sum over the integrals from $P_0$ to all $Q \in\text{pr}_x^{-1}(0)$ vanishes modulo the period lattice $\Lambda$ (in fact this is true for any non-branch point).
    Namely, for every $\wtij \in \W$ we have
    \begin{align*}
      & \sum_{Q \in\text{pr}_x^{-1}(0)} \int_{P_0}^Q \wtij =  \sum_{l=0}^{m-1} \int_{P_0}^{(0,\zeta^l \sqrt[m]{f(0)})} \wtij \\
    & = m \int_{P_0}^{P_k} \wtij + \left(1 + \zeta^{-j} + \dots + \zeta^{-j(m-1)}\right) \int_{P_0}^{(0,\sqrt[m]{f(0)})} \wtij
    \end{align*}
    for some $k \in \{1,\dots,n\}$ and therefore
    \begin{equation*}
       \sum_{Q \in\text{pr}_x^{-1}(0)} \int_{P_0}^Q \bar\w =  m \int_{P_0}^{P_k} \bar\w \overset{\eqref{m-eq:m_tors}}{\equiv} 0 \mod \Lambda.
    \end{equation*}
    If we take $Q_i^{(s)} = \Phi(\zeta_{\delta}^s,t_i^{(s)}) \in \cu \setminus \mathcal{P}$, $i = 1,\dots,d-1$, we are done:
    \begin{align*}
    -\AJ( [\sum_{i = 1}^{d-1} \Phi(\zeta_{\delta}^s,t_i^{(s)}) - N \sum_{Q \in\text{pr}_x^{-1}(0)} Q ] ) \equiv
    -\sum_{i = 1}^{n-1} \int_{P_0}^{Q_i^{(s)}} \bar\w \mod \Lambda.
    \end{align*}
     \end{proof}

    In the case of Theorem \ref{m-thm:ajm_inf_ord1} there exist additional relations between the vector integrals in
    \eqref{m-eq:ajm_inf_ord1} which we are going to establish now.
    Let $s \in \{1,\dots,\delta\}$, fix $i \in \{1,\dots,n-1\}$ and denote $Q^{(s)} = Q_i^{(s)}$ and $t^{(s)} = t_i^{(s)}$.
    On $\cafft$ we have the relation
     \begin{equation*}
	( \zeta_{\delta}^s,t^{(s)} ) = ( \zeta_{\delta}^s, \zeta_{\delta}^{-\nu s} t^{(\delta)} )
     \end{equation*}
     and therefore, if we write $(x^{(s)},y^{(s)}) :=  \Phi( \zeta^s,t^{(s)})$, then
     \begin{equation*}
      Q^{(s)} = (x^{(s)},y^{(s)}) = (x^{(\delta)},\zeta_{\delta}^{(\mu+\nu N)s} y^{(\delta)})).
     \end{equation*}
     The $Q^{(s)}$ having identical $x$-coordinates implies that there exists a $k \in \{1,\dots,n\}$ such that
     \begin{equation*}
      \int_{P_0}^{Q^{(s)}} \bar\w =  \int_{P_0}^{P_k} \bar\w  +  \int_{P_k}^{Q^{(s)}}  \bar\w,
     \end{equation*}
     while the relation between their $y$-coordinates yields
     \begin{equation*}
       \int_{P_k}^{Q^{(s)}} \wtij = \zeta_{\delta}^{-(\mu+\nu N)sj} \int_{P_k}^{Q^{(\delta)}} \wtij
     \end{equation*}
     for all $\wtij \in \W$. This proves the following corollary:

    \begin{coro}\label{coro:ajm_inf_ord1}
    Under the assumptions of Theorem \ref{m-thm:ajm_inf_ord1} and with the above notation we can obtain the
    image of $D_{\infty}^{(s)}$ under the Abel-Jacobi map
    for all $s = 1,\dots,\delta$ from the $n-1$ vector integrals
    \begin{equation*}
     \int_{P_k}^{Q_i^{(\delta)}} \bar\w, \quad i = 1,\dots,n-1.
    \end{equation*}
    \end{coro}

    \bigskip

    Unfortunately, this is just a special case. If $\ord_{P_{\infty}^{(s)}}(D_1)$ is greater than $1$ (for instance, if $\delta \ne m$),
     the vertical line defined by $r-\zeta_{\delta}^s$ is tangent to the curve $\cafft$
    at $(\zeta_{\delta}^s,0)$ and cannot be used for our purpose.

    Consequently, we must find another function. One possible choice here is the line defined by $r - t -  \zeta_{\delta}^s$, which
    is now guaranteed to have a simple intersection with $\cafft$
    at $(\zeta_{\delta}^s,0)$ and does not intersect $\cafft$ in $(\zeta_{\delta}^{s'},0)$, $s \ne s'$.

    The corresponding principal divisor is given by
     \begin{equation*}
      E_2 = \div(r - t - \zeta_{\delta}^s) = \sum_{i = 1}^{d} (t_i^{(s)}+\zeta_{\delta}^s,t_i^{(s)}) - \nu \sum_{k=1}^n
      \Phi^{-1}(x_k,0)- N E_2',
   \end{equation*}
      where the $t_i^{(s)}$ are the zeros of $h(t) = \prod_{k=1}^n (1-x_k(t+\zeta_{\delta}^{(s)})^{\nu}t^M) - 1 \in \C[t]$, $d = \deg(h)$ and
    \begin{equation}\label{eq:zero_bp2}
       E_2' = \begin{cases}
	   (m-\frac{M+\nu}{N}) \Phi^{-1}(0,0), \quad \text{if} \; 0 \in X, \\
             \sum_{Q \in\text{pr}_x^{-1}(0)} \Phi^{-1}Q, \quad \text{otherwise.}
             \end{cases}
    \end{equation}
    Now,
    \begin{equation*}
     D_2 := \div \left( \frac{y^M}{x^N} - \frac{1}{x^{\mu}y^{\nu}} - \zeta_{\delta}^s \right)
    \end{equation*}
    is a principal divisor on $\caff$ such that
    $\ord_{P_{\infty}^{(s)}}(D_2) = 1$.

    \begin{thm}\label{thm:ajm_inf_ordgt1}
      Assume $\ord_{P_{\infty}^{(s)}}(D_1) > 1$ and $0 \not\in X$. Then, for $s = 1,\dots,\delta$, there exist points $Q_1^{(s)},\dots,Q_{d-1}^{(s)} \in \cu \setminus \mathcal{P}$ such that
    \begin{equation*}
       \AJ( [D_{\infty}^{(s)}] ) \equiv -\sum_{i = 1}^{d-1} \int_{P_0}^{Q_i^{(s)}} \bar\w + \nu \sum_{k=1}^n
      \int_{P_0}^{P_k} \bar\w \mod \Lambda,
    \end{equation*}
    where $d = n(\nu+M)$.
    \end{thm}
   \begin{proof}
    First note that $0 \not\in X$ implies $d = \deg(h) = n(\nu+M)$. Moreover, our assumption implies $\ord_{P_{\infty}^{(s)}}(D_2) = 1$ so that we may assume
    $t_d^{(s)} = 0$ and $t_i^{(s)} \ne 0$ for $i=1,\dots,d-1$. 	Then,
    \begin{align*}
     \AJ( [D_{\infty}^{(s)}] ) \equiv & \AJ([D_{\infty}^{(s)} - D_2]) \\
      \equiv & -\AJ( [\sum_{i = 1}^{d-1} \Phi(t_i^{(s)}+\zeta_{\delta}^s,t_i^{(s)}) - \nu \sum_{k=1}^n
      (x_k,0) - N \sum_{Q \in\text{pr}_x^{-1}(0)} Q ] )
       \mod \Lambda.
    \end{align*}
   Choosing the points
    $Q_i^{(s)} = \Phi(t_i^{(s)}+\zeta_{\delta}^s,t_i^{(s)}) \in \cu \setminus \mathcal{P}$ and
     using the same reasoning as in the proof of Theorem \ref{m-thm:ajm_inf_ord1}
    proves the statement.
     \end{proof}

    \begin{rmk}\label{rmk:zero_bp}
     We can easily modify the statements of the Theorems \ref{m-thm:ajm_inf_ord1} and \ref{thm:ajm_inf_ordgt1} to hold for
     $0 \in X$, i.e. when $0$ is a branch point.
     Using equation \eqref{m-eq:zero_bp1}, the statement of Theorem \ref{m-thm:ajm_inf_ord1} becomes
      \begin{equation*}
       \AJ( [D_{\infty}^{(s)}] ) \equiv -\sum_{i = 1}^{n-1} \int_{P_0}^{Q_i^{(s)}} \bar\w + N(m-M) \int_{P_0}^{(0,0)} \bar\w \mod \Lambda,
    \end{equation*}
    whereas, using equation \eqref{m-eq:zero_bp2}, the statement of Theorem \ref{m-thm:ajm_inf_ordgt1} becomes
      \begin{equation*}
       \AJ( [D_{\infty}^{(s)}] ) \equiv -\sum_{i = 1}^{d-1} \int_{P_0}^{Q_i^{(s)}} \bar\w + \nu \sum_{k=1}^n
      \int_{P_0}^{P_k} \bar\w + (Nm-M-\nu) \int_{P_0}^{(0,0)} \bar\w \mod \Lambda,
    \end{equation*}
    with $d = n(\nu+M)$.
    \end{rmk}

  \subsection{Reduction modulo period lattice}\label{subsec:lat_red}

    In order for the Abel-Jacobi map to be well defined we have to reduce modulo the period lattice $\Lambda =
  \Omega\Z^{2g}$, where $\Omega = (\OA, \OB)$ is the big period matrix, computed as explained in
  Section \ref{m-sec:strat_pm}.

   Let $v = \int_P^Q \bar\w \in \C^g$ be a vector obtained by integrating the holomorphic differentials in
   $\W$.
   We identify $\C^g$ and $\R^{2g}$ via the bijection
   \begin{equation*}
    \iota: v = (v_1,\dots,v_g)^T \mapsto (\Re(v_1),\dots,\Re(v_g),\Im(v_1),\dots,\Im(v_g))^T.
   \end{equation*}
    Applying $\iota$ to the columns of $\Omega$ yields the invertible real matrix
   \begin{equation*}
    \Omega_{\R} =
   \begin{pmatrix}
     \Re(\OA) & \Re(\OB) \\
     \Im(\OA) & \Im(\OB)
    \end{pmatrix} \in \R^{2g \times 2g}.
   \end{equation*}
   Now, reduction of $v$ modulo $\Lambda$ corresponds bijectively to taking the fractional part of $\Omega_{\R}^{-1}\iota(v)$
   \begin{equation*}
    v \bmod \Lambda \leftrightarrow \lfloor \Omega_{\R}^{-1}\iota(v) \rceil.
   \end{equation*}


  \section{Computational aspects}\label{sec:comp_asp}

   \subsection{Complexity analysis}

   We recall the parameters of the problem: we consider a superelliptic curve $\cu$ given by
   $\caff:y^m=f(x)$ with $f \in \C[x]$ separable of degree $n$. The genus $g$ of $\cu$ satisfies
   $$g \leq \frac{(m-1)(n-1)}2=O(mn).$$

   Let $D$ be some desired accuracy (a number of decimal digits). The computation of
   the Abel-Jacobi map on $\cu$ has been decomposed into the
   following list of tasks:
   \begin{enumerate}
       \item computing the $(n-1)$ vectors of elementary integrals,
       \item computing the big period matrix $\Omega=(\OA,\OB)$ \eqref{m-eq:OAOB},
       \item computing the small period matrix $\tau = \OA^{-1}\OB$ \eqref{m-eq:tau},
       \item evaluating the Abel-Jacobi map at a point $P \in \cu$,
   \end{enumerate}
   all of these to absolute precision $D$.

   Let $N(D)$ be the number of points of numerical integration.
   If $m=2$, we have
   $N(D)=O(D)$ using Gauss-Chebychev integration, while $N(D)=O(D\log D)$
   via double-exponential integration.

   For multiprecision numbers, we consider (see \cite{BrentZimmermann}) that the multiplication has
   complexity $\cmul(D)=O(D \log^{1+\varepsilon}D)$,
   while simple transcendental functions (log, exp, tanh, sinh,\dots) can be evaluated
   in complexity $\ctrig(D)=O(D\log^{2+\varepsilon} D$).
    Moreover, we assume that multiplication of a $g \times g$ matrix can
    be done using $O(g^{2.8})$ multiplications.

   \subsubsection{Computation of elementary integrals}
   \label{sec:comp_elem}

   For each elementary cycle $\gamma_e\in \Gamma$, we numerically evaluate the vector of $g$
   elementary integrals from \eqref{m-eq:elem_num_int} as sums of the form
   \begin{equation*}
       I_{a,b} \approx \sum_{k=1}^N w_k\frac{u_k^{i-1}}{y_k^j},
   \end{equation*}
   where $N = N(D)$ is the number of integration points, $w_k,u_k$ are integration weights and points,
   and $y_k=\ytab(u_k)$.

   We proceed as follows:
   \begin{itemize}
   \item for each $k$, we evaluate the absissa and weight $u_k,w_k$ using
       a few \footnote{this can be reduced to evaluating
           a few multiplications and at most one exponential.} trigonometric or hyperbolic functions,
   \item we compute $y_k=\ytab(u_k)$ using $n-2$ multiplications and one $m$-th root,
       as shown in \S \ref{m-subsec:computing_roots} below;
   \item starting from $\frac{w_k}{y_k}$, we evaluate all $g$ terms $w_k\frac{u_k^{i-1}}{y_k^j}$
       each time either multiplying by $u_k$ or by $\frac{1}{y_k}$, and add each to the corresponding
       integral.
   \end{itemize}

   Altogether, the computation of one vector of elementary integrals takes
   \begin{equation}\label{eq:elem_int_complexity}
    \ctot(D) = N(D)\ctrig(D)+N(D)(n-1)\cmul(D)+N(D)g\cmul(D)
   \end{equation}
    operations,
   so that depending on the integration scheme we obtain:
   \begin{thm}\label{thm:complexity_integrals}
       Each of the $(n-1)$ elementary vector integrals can be computed to precision $D$ using
       \begin{equation*}
           O(N(D)\cmul(D)(g+\log D)) =
           \begin{cases}
               O(D^2\log^{1+\varepsilon} D (g + \log D)) \text{ operations, if $m=2$,}\\
               O(D^2\log^{2+\varepsilon} D (g + \log D)) \text{ operations, if $m>2$.}
           \end{cases}
       \end{equation*}
   \end{thm}

   \subsubsection{Big period matrix}

   One of the nice aspects of the method is that we never compute
   the dense matrix $\OC\in\C^{g\times 2g}$ from \eqref{m-eq:OC}, but
   keep the decomposition of periods in terms of the elementary integrals
   $\int_{\gamma_e}\omega_{i,j}$ in $\C^{g\times (n-1)}$.

   Using the symplectic base change matrix $S$ introduced
   in \S \ref{m-subsec:symp_basis}, the symplectic homology basis is given
   by cycles of the form
   \begin{equation}
       \label{eq:base_change_cycles}
       \alpha_i = \sum_{\substack{e \in E \\ l\in\Z/m\Z}} s_{e,l}\gamma_e^{(l)}
   \end{equation}
   where $\gamma_e^{(l)} \in \Gamma$ is a generating cycle
   and $s_{e,l}\in\Z$ is the corresponding entries of $S$.

   We use \eqref{m-eq:periods} to compute the coefficients of the big period
   matrix $(\OA,\OB)$, so that each term of \eqref{m-eq:base_change_cycles}
   involves only a fixed number of multiplications.

   In practice, these sums are sparse and their coefficients are very small integers
   (less than $m$), so that the change of basis is performed using
   $O(g^3D\log^{1+\varepsilon}D)$ operations
   (each of the $O(g^2)$ periods is a linear combination of $O(g)$ elementary integrals,
   the coefficients involving precision $D$ roots of unity).

   However, we have no proof of this fact and in general the symplectic reduction
   could produce dense base change with coefficients of size $O(g)$,
   so that we state the far from optimal result
   \begin{thm}
       Given the $(n-1)\times g$ elementary integrals to precision $D$,
       we compute the big period matrix using $O(g^3(D+g)\log(D+g))$ operations.
   \end{thm}

   \subsubsection{Small period matrix}

   Finally, the small period matrix is obtained by solving $\OA\tau=\OB$,
   which can be done using $O(g^{2.8})$ multiplications.

   \subsubsection{Abel-Jacobi map}

  This part of the complexity analysis is based on the results of Section \ref{m-sec:comp_ajm} and assumes that we have already computed a big period matrix and all related data.

  Let $\ctot(D)$ be the number of operations needed to compute a vector of $g$ elemenatary integrals  (see \eqref{m-eq:elem_int_complexity}). The complexity class of $\ctot(D)$ in $O$-notation is given in
  Theorem \ref{m-thm:complexity_integrals}.

   \begin{thm} \
   \begin{itemize}
     \item[{\upshape{(i)}}] For each finite point $P \in \caff$ we can compute $\int_{P_0}^P \bar\w$ to precision $D$ using
      $\ctot(D)$ operations.
     \item[{\upshape{(ii)}}] For each infinite point $P_{\infty} \in \cu$ we can compute a representative of $\int_{P_0}^{P_{\infty}} \bar\w \mod \Lambda$ to precision $D$ using
      \begin{itemize}
       \item[$\bullet$] $n$ vector additions in $\C^g$, if $\delta = \gcd(m,n) = 1$,
       \item[$\bullet$] $n \ctot(D)$ operations in the case of Theorem \ref{m-thm:ajm_inf_ord1},
       \item[$\bullet$] $n(n+\frac{m}\delta)\ctot(D)$ operations  in the case of Theorem \ref{m-thm:ajm_inf_ordgt1}.
      \end{itemize}
      \item[{\upshape{(iii)}}] Reducing a vector $v \in \C^g$ modulo $\Lambda$ can be done using $O(g^{2.8})$ multiplications.
    \end{itemize}
  \end{thm}
   \begin{proof}
    \begin{itemize}
     \item[(i)] Follows from combining the results from \S \ref{m-subsec:ajm_ram_pts} and Remark \ref{m-rmk:ajm_finite_int}.
     \item[(ii)] The statements follow immediately from \S \ref{m-subsec:ajm_inf_cop}, Theorem \ref{m-thm:ajm_inf_ord1} and Theorem \ref{m-thm:ajm_inf_ordgt1}.
     \item[(iii)] By \S \ref{m-subsec:lat_red}, the reduction modulo the period lattice requires one $2g \times 2g$ matrix inversion and one multiplication.
    \end{itemize}
   \end{proof}

   \subsection{Precision issues}

   As explained in \S \ref{subsec:arb}, the ball arithmetic model allows
   to certify that the results returned by the Arb program \cite{Johansson2013arb} are correct.
   It does not guarantee that the result actually achieves the desired
   precision.

   As a matter of fact, we cannot prove a priori that bad accuracy loss
   will not occur while summing numerical integration terms or
   during matrix inversion.

   However, we take into account all predictable loss of precision:
   \begin{itemize}
       \item While computing the periods using equations \eqref{eq:periods}
           and \eqref{eq:polshift}, we compute a sum with coefficients
           \begin{equation*}
               C_{a,b}^{-j} \left(\frac{b-a}{2}\right)^i
               {i-1 \choose l} \left(\frac{b+a}{b-a}\right)^{i-1-l}
           \end{equation*}
           whose magnitude can be controlled a priori. It has size $O(g)$.
       \item The size of the coefficients of the symplectic reduction
           matrix are tiny (less than $m$ in practice), but we can take
           their size into account before entering the numerical steps.
           Notice that generic HNF estimates lead to a very pessimistic
           estimate of size $O(g)$ coefficients.
       \item Matrix inversion of size $g$ needs $O(g)$ extra bits.
   \end{itemize}
   This leads to increasing the internal precision from $D$ to $D + O(g)$,
   the implied constant depending on the configuration of branch points.

   \begin{rmk}
   In case the end result is imprecise by $d$ bits, the user simply needs to
   run another instance to precision $D+d$.
   \end{rmk}

   \subsection{Integration parameters}

   \subsubsection{Gauss-Chebychev case}

   Recall from \S \ref{m-subsec:gauss_chebychev_integration} that we can parametrize the ellipse $ε_r$ via
   $$ε_r = \set{ \cosh(r+it) = \cos(t-ir), t\in]-π,π] }.$$

   \begin{figure}[H]
       \begin{center}
       \includegraphics[width=5cm,page=3]{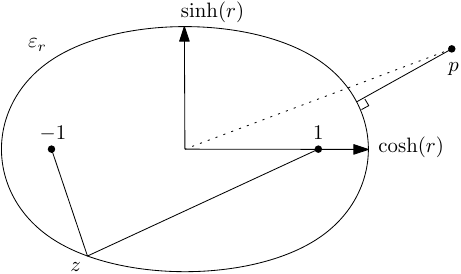}
   \end{center} \caption{ellipse parameters.}
   \label{fig:ellipse2}
   \end{figure}

   The sum of its semi-axes is $e^{r}$
   and one needs
   \[
       N \geq \frac{D+\log(2πM(r)+e^{-D})}{2r}
   \]
   to have $\abs{E(N)}\leq e^{-D}$.

   The distance $d_k=\dist(u_k,ε_r)$ from a branch point $u_k$
   to the ellipse $ε_r$ can be computed
   applying Newton's method to the scalar product function
   $s(t) = \Re(\overline{z'}(u_k-z))$, where $z = \cos(t-ir)$ and
   we take $t=\Re(\arccos(u_k))$ as a starting point.
   By convexity of the ellipse,
   the solution is unique on the quadrant containing $u_k$.

   \paragraph{Choice of $r$}\label{par:gc_int_r}

   Let $\abs{u_k-1}+\abs{u_k+1}=2\cosh(r_k)$. We need to choose
   $r<r_0=\min_k r_k$ (so that $u_k\not\in ε_r$) in order to minimize
   the number of integration points \eqref{m-eq:Ngc}. We first
   estimate how the bound $M(r)$ varies for $r<r_0$.
   \begin{itemize}
       \item
   For all $k$ such that $r_k > r_0$, we compute
   explicitly the distance $d_k=\dist(u_k,ε_{r_0})<\dist(u_k,ε_r)$.
   \item
   For $k$ such that $r_k=r_0$, we use first order approximation
   \[ \dist(u_k,Z_{r-η}) = η D_k + O(η^2) \],
   where $D_k = \abs{\frac{\partial u_k}{\partial r_k}} = \abs{\sin(t_k-ir_k)}$.
   \end{itemize}

   Let $K$ be the number of branch points $u_k$ such that $r_k=r_0$ and
   \[ M_0 = \sqrt{\prod_{r_k = r_0} D_k\prod_{r_k>r_0}d_k}^{-1}, \]
   then the integrand is bounded on $ε_{r_0-η}$ by
   \[ M(r_0-η) = M_0 \sqrt{η}^{-K} (1+O(η)). \]
   Plugging this into \eqref{eq:Ngc}, the number of integration points
   satisfies
   \[
       2N = \frac{D+\log(2πM_0) - K/2 \log(η) }{r_0-η}(1+O(η)).
   \]

   The main term is minimized for $η$ satisfying
   $η\left(2\frac{D+\log(2πM_0)}K+1-\log(η)\right)=r_0$. The solution
   can be written as a Lambert function or we use
   the approximation
   \[ r = r_0 - η = r_0 \left( 1 - \frac{1}{A+\log\frac{A}{r_0}} \right), \]
   where $A = 1+\frac2K(D+\log(2πM_0))$.

   \subsubsection{Double-exponential case}\label{subsec:de_case}

   For the double-exponential integration (\S \ref{m-subsec:de_int})
   we use the parametrization
   $$\partial Z_r = \set{ z = \tanh(λ\sinh(t+ir)), t\in\R }$$ to compute
   the distance from a branch point $u_k$ to $Z_r$ by Newton's method
   as before.

   Unfortunately, the solution may not be unique, so once
   the parameter $r<r_0$ is chosen (see below), we use ball arithmetic to compute a rigorous
   bound of the integrand on the boundary of $Z_r$. The process consists in
   recursively subdividing the interval until the images of the subintervals by the
   integrand form an $ε$-covering.

   \paragraph{Choice of $r$}

   We adapt the method used for Gauss-Chebychev. This time the number $N$ of integration
   points is obtained from equation \eqref{m-eq:de_parameters}.

   Writing $u_k = \tanh(λ\sinh(t_k+ir_k))$, we must choose
   $r<r_0=\min_k \{r_k\}$ to ensure $u_k\not\in Z_r$. Let
   \[ M_0 = (\prod_{r_k = r_0} D_k\prod_{r_k>r_0}d_k)^{-j/m} \]
   where $d_k=\dist(u_k,Z_{r_0})<\dist(u_k,Z_r)$ and
   \[ D_k = \abs{\frac{\partial u_k}{\partial r_k}} = \abs{\frac{λ \cosh(t_k+ir_k)}{\cosh(λ\sinh(t_k+ir_k))^2}} \]
   is such that $\dist(u_k,Z_{r-η}) = η D_k + O(η^2)$, then
   the integrand is bounded on $Z_{r_0-η}$ by
   \[ M_2 = M_0 η^{-\frac{jK}m} (1+O(η)). \]
   Then
   \[ h = \frac{2π(r_0-η)}{D+\log(2B(r_0,α)M_0)-jK/m\log(η)}+O(η) \]
   and the maximum is obtained for $η$ solution of $η(A-\log η)=r_0$
   where $A=1+\frac{m}{jK}(D+\log(2B(r_0,α)M_0))$.

   \subsection{Implementation tricks}

   Here we simply give some ideas that we used in our implementation(s) to improve constant factors hidden in the big-$O$ notation, i.e. the absolute running time.

   In practice, 80 to 90\% of the running time is spent on numerical integration
   of integrals \eqref{m-eq:periods}. According to \S\ref{m-sec:comp_elem},
   for each integration point $u_k\in]-1,1[$ one first evaluates the $y$-value
   $y_k=\ytab(u_k)$, then adds the contributions $w_k\frac{u_k^i}{y_k^j}$ to
   the integral of each of the $g$ differential forms.

   We shall improve on these two aspects, the former being prominent for hyperelliptic curves,
   and the latter when the $g \gg n$.

    \subsubsection{Computing products of complex roots}\label{subsec:computing_roots}

    Following our definition \eqref{m-eq:def_yab}, computing $\ytab(u_k)$ involves
    $(n-2)$ $m$-th roots for each integration point.

    Instead, we fall back to one single (usual) $m$-th root
    by computing $q(u)\in\frac12\Z$ such that
  \begin{equation}
      \label{eq:ytab_comp}
      \ytab(u) = \zeta^{q(u)} \Big( \prod_{u_k\in U^-}(u-u_k) \prod_{u_k\in U^+} (u_k-u) \Big)\mr.
  \end{equation}
  This can be done by tracking
  the winding number of the product while staying away from the branch cut
  of the $m$-th root.
  For complex numbers $z_1,z_2 \in \C$ we can make a diagram of
  $\frac{\sqrt[m]{z_1}\sqrt[m]{z_2}}{\sqrt[m]{z_1z_2}} \in \{ 1, \zeta,
  \zeta^{-1} \}$, depending on the position of $z_1,z_2$ and their product
  $z_1z_2$ in the complex plane, resulting in the following lemma:

  \begin{lemma}\label{lemma:wind_numb}
  Let $z_1,z_2 \in \C  \setminus  ]\infty,0]$. Then,
  $$\frac{\sqrt[m]{z_1}\sqrt[m]{z_2}}{\sqrt[m]{z_1z_2}} = \begin{cases}
                                                           \zeta, \quad \text{if} \quad \Im(z_1), \Im(z_2) > 0 \quad \text{and} \quad \Im(z_1z_2) < 0 , \\
                                                           \zeta^{-1}, \text{if} \quad \Im(z_1), \Im(z_2) < 0 \quad \text{and} \quad \Im(z_1z_2) > 0 , \\
                                                           1, \quad \text{otherwise}.
                                                         \end{cases}$$
   For $z \in ]\infty,0]$ we use $\sqrt[m]{z} = \zeta^{\frac{1}{2}} \cdot \sqrt[m]{-z}$.
  \end{lemma}
  \begin{proof}
   Follows from the choices for $\sqrt[m]{\cdot}$ and $\zeta$ that were made in \S \ref{m-subsec:roots_branches}.
  \end{proof}
  Lemma \ref{lemma:wind_numb} can easily be turned into an algorithm that computes $q(u)$.

   \subsubsection{Doing real multiplications}\label{subsec:real_mult}

   Another possible bottleneck comes from the multiplication by the numerator
   $u_k$, which is usually done $g-m-1$ times for each of
   the $N$ integration points (more precisely, as we saw in the proof of Proposition \ref{m-prop:holom_diff}, for each exponent $j$
   we use the exponents $0\leq i\leq n_i = \floor{\frac{nj-δ}m}$, with $\sum n_i = g$).

   Without polynomial shift \eqref{eq:polshift}, this numerator should be
   $x_k=u_k+\frac{b+a}{b-a}$. However, $x_k$ is a complex number while $u_k$
   is real, so computing with $u_k$ saves a factor almost $2$ on this aspect.

   \subsection{Further ideas}

   \subsubsection{Improving branch points}

   As we saw in Section \ref{m-sec:numerical_integration}, the number of integration points
   closely depends on the configuration of branch points.

   In practice, when using double-exponential integration, the constant $r$ is usually bigger than $0.5$
   for random points, but we can exhibit bad configurations with $τ\approx 0.1$.
   In this case however, we can perform a change of coordinate by a Moebius transform
   $x\mapsto \frac{ax+b}{cx+d}$ as explained in Remark \ref{m-rmk:moebius} to redistribute the points more evenly.

   Improving $τ$ from $0.1$ to say $0.6$ immediately saves a factor $6$ on the running time.

   \subsubsection{Near-optimal tree}
    As explained in \S \ref{m-subsec:cycles_homo} we integrate along the edges of a maximal-flow spanning tree $T = (X,E)$, where the capacity $r_e$ of an edge $e = (a,b) \in E$ is computed as
    \begin{equation*}
     r_e =  \min_{c \in X \setminus \{ a,b\}} \begin{cases}
         \frac{\abs{c-a}+\abs{c-b}}{\abs{b-a}}, \text{ if $m=2$,}\\
         \abs{ \Im(\sinh^{-1}(\tanh^{-1}(\frac{2c - b -a}{b-a})/\lambda) }, \text{ if $m>2$.}
            \end{cases}
    \end{equation*}
    Although this can be done in low precision, computing $r_e$ for all $(n-1)(n-2)/2$ edges of the complete graph
    requires $O(n^3)$ evaluation of elementary costs (involving transcendantal functions if $m>2$).

    For large values of $n$ (comparable to the precision), the computation of
    these capacities has a noticable impact on the running time. This can be
    avoided by computing a \emph{minimal spanning tree} that uses the euclidean
    distance between the end points of an edge as capacity, i.e.
    $r_e = \abs{b-a}$, which reduces the complexity to $O(n^2)$
    multiplications.

    Given sufficiently many branch points that are randomly distributed in the
    complex plane, the shortest edges of the complete graph tend to agree with
    the edges that are well suited for integration.

    \subsubsection{Taking advantage of rational equation}

    In case the equation \eqref{eq:aff_model} is given by a polynomial $f(x)$
    with small rational coefficients, one can still improve the computation
    of $\ytab(u)$ in \eqref{eq:ytab_comp} by going back to
    the computation of $y(\xab(u))=f(x)^{\frac1m}$.
    The advantage is that baby-step giant-step splitting can be used for
    the evaluation of $f(x)$, reducing the number of multiplications to
    $O(\sqrt{n})$. In order to recover $\ytab(u)$, one needs to divide by
    $\sqrt[m]{1-u^2}$ and adjust a multiplicative constant including
    the winding number $q(u)$, which can be evaluated at low precision.
    This technique must not be used when $u$ gets close to $\pm1$.

    \subsubsection{Splitting bad integrals or moving integration path}

    Numerical integration becomes very bad when there are other branch points
    relatively close to an edge. The spanning tree optimization does not help
    if some branch points tend to cluster while other are far away.
    In this case, one can always split the bad integrals to improve the
    relative distances of the singularities. Another option with
    double exponential integration is to shift the integration path.

  \section{Examples and timings}\label{sec:examples_timings}

    \newcommand{\bern}{\mathcal{B}}
  \newcommand{\bernrev}{\widetilde{\bern}}

  For testing purposes we consider a family of curves given by Bernoulli polynomials
          \begin{equation*}
              \bern_{m,n} : y^m = B_n(x) = \sum_{k=0}^n\binom nkb_{n-k}x^k
          \end{equation*}
  as well as their reciprocals
          \begin{equation*}
              \bernrev_{m,n} : y^m = x^nB_n\left(\frac1x\right).
          \end{equation*}

  The branch points of these curves present interesting patterns which can be respectively considered
  as good and bad cases from a numerical integration perspective (Figure \ref{m-fig:roots_bern}).

  \begin{figure}[H]
      \begin{center}
          \subfloat[$\bern_{m,8}$]{ \includegraphics[width=.5\linewidth,page=1]{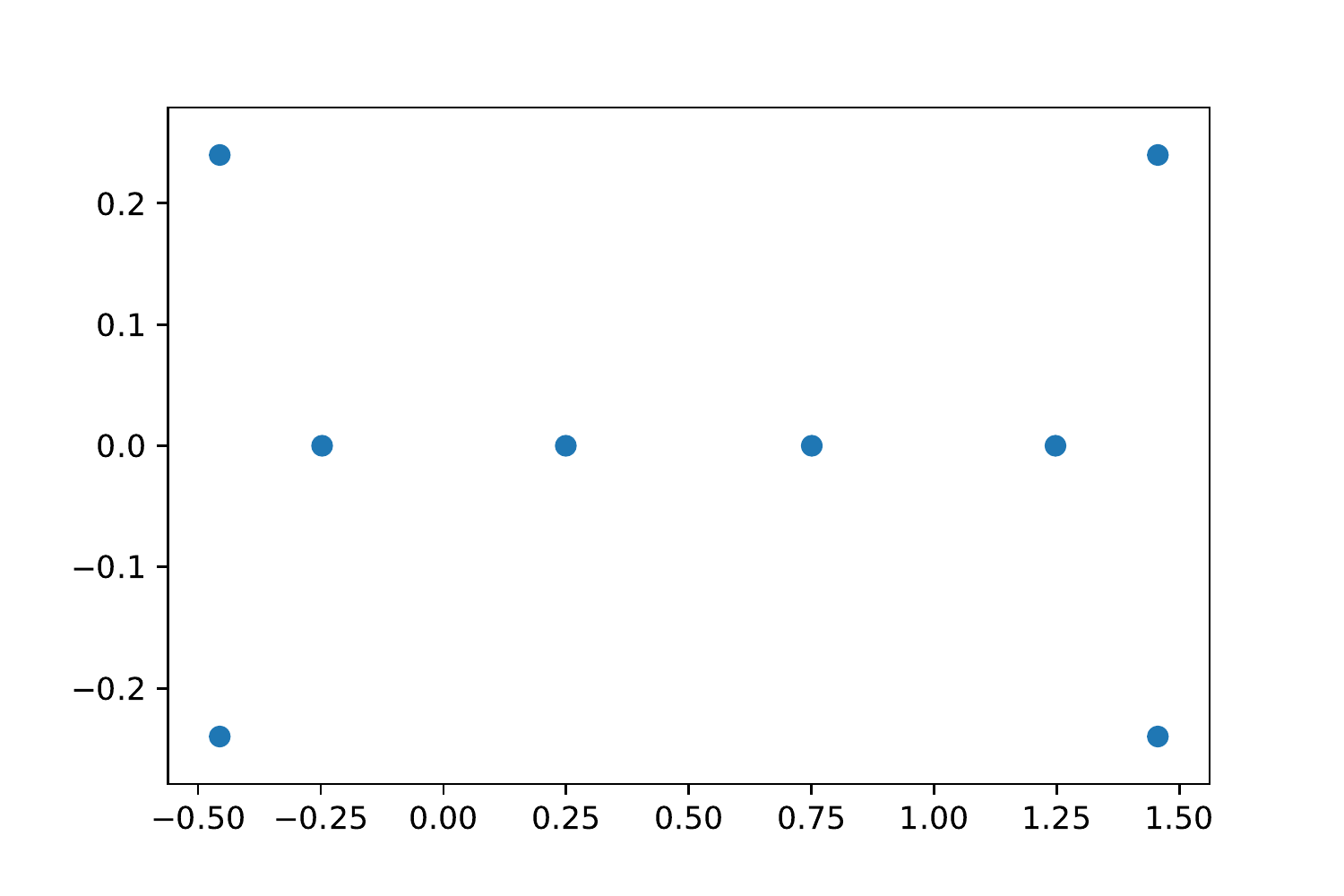} }
          \subfloat[$\bern_{m,30}$]{ \includegraphics[width=.5\linewidth,page=2]{roots_bern.pdf} }\\
          \subfloat[$\bernrev_{m,8}$]{ \includegraphics[width=.5\linewidth,page=3]{roots_bern.pdf} }
          \subfloat[$\bernrev_{m,30}$]{ \includegraphics[width=.5\linewidth,page=4]{roots_bern.pdf} }
      \end{center}
      \caption{configurations of branch points.}
  \label{fig:roots_bern}
  \end{figure}

  In the case of hyperelliptic curves, we compare our timings with the existing
  Magma code \cite{vanWamelen06}. We obtain a huge speedup which is mostly due
  to the better integration scheme, but more interesting is the fact that the
  running time of our algorithm mainly depends on the genus and the precision,
  while that of Magma depends a lot on the branch points and behaves very badly
  in terms of the precision.

  \begin{table}[H]
      \begin{center}
          \begin{tabular}{lllrrrrr}
              \toprule
              & & \hfill bits & 128 & 512 & 2000 & 4000 & 10000 \\
              genus & curve & \hfill digits & 38 & 154 & 600 & 1200 & 3000 \\
              \midrule
              3 & $\bern_{2,8}$
              &   Arb         & 5e-3  & 0.01  & 0.16    & 0.48     & 3.99 \\
              & & Magma (new) & 0.05  & 0.08  & 0.44   & 2.16    & 25.3 \\
              & & Magma (old) & 0.33  & 0.44  & 6.28   & 421  & --- \\
              \cmidrule{3-8}
              & $\bernrev_{2,8}$
              &   Arb         & 5e-3  & 0.01  & 0.17    & 0.54      & 4.58 \\
              & & Magma (new) & 0.06  & 0.11  & 0.67   & 3.42     & 40.6 \\
              & & Magma (old) & 0.42  & 0.45  & 6.44   & 457   & --- \\
              \midrule
              14 & $\bern_{2,30}$
                & Arb         & 0.05  & 0.22   & 1.99    & 8.74     & 80.9 \\
              & & Magma (new) & 0.55  & 0.94  & 4.64   & 18.7   & 185.1 \\
              & & Magma (old) & 5.15  & 10.1 & 134 & 9291 & --- \\
              \cmidrule{3-8}
              & $\bernrev_{2,30}$
              &   Arb         & 0.05  & 0.23   & 2.11    & 9.31      & 87.8 \\
              & & Magma (new) & 0.51  & 1.02  & 5.40   & 21.9    & 227 \\
              & & Magma (old) & 14.8  & 42.6  & 370 & 12099 & --- \\
              \midrule
              39 & $\bern_{2,80}$
              &   Arb         & 0.69 & 1.64   & 16.1   & 70.5    & 601 \\
              & & Magma (new) & 6.29  & 9.08  & 36.4  & 122  & 1024 \\
              \bottomrule
          \end{tabular}
          \caption{timings for hyperelliptic curves, single core Xeon E5 3GHz (in seconds).}
      \end{center}
  \end{table}

  \begin{table}[H]
      \begin{center}
          \begin{tabular}{lllrrrrr}
              \toprule
              & & \hfill bits & 128 & 512 & 2000 & 4000 & 10000 \\
              genus & curve & \hfill digits & 38 & 154 & 600 & 1200 & 3000 \\
              \midrule
              21 & $\bern_{7,8}$
              &   Arb         & 0.06 & 0.27   & 4.25    & 29.5    & 455 \\
              & & Magma (new) & 0.23  & 1.06  & 14.6  & 83.1   & 1035 \\
              \cmidrule{3-8}
              & $\bernrev_{7,8}$
              &   Arb         & 0.03 & 0.19  & 7.44    & 58.8   & 1027 \\
              & & Magma (new) & 0.30  & 1.64  & 23.9  & 132 & 1613 \\
              \midrule
              84 & $\bern_{25,8}$
              &   Arb         & 0.09   & 0.45   & 8.86     & 55.6    & 727 \\
              & & Magma (new) & 0.74   & 2.60   & 27.2   & 135  & 1529 \\
              \midrule
              87 & $\bern_{7,30}$
              &   Arb         & 2.05  & 6.46   & 43.9   & 249   & 3091 \\
              & & Magma (new) & 2.29  & 10.0 & 93.8  & 461  & 4990 \\
              \midrule
              348 & $\bern_{25,30}$
              &   Arb         & 2.82   & 9.57    & 101  & 557   & 6195 \\
              & & Magma (new) & 19.9  & 41.4  & 234  & 1014 & 9614 \\
              \midrule
              946 & $\bern_{25,80}$
              &   Arb         & 67.8  & 182  & 952   & 4330 \\
              & & Magma (new) & 369 & 585 & 2132 & 7474 \\
              \bottomrule
          \end{tabular}
          \caption{timings for superelliptic curves, single core Xeon E5 3GHz (in seconds).}
      \end{center}
  \end{table}

  \newpage

  \section{Outlook}\label{sec:outlook}

  In this paper we presented an approach based on numerical integration for
  multiprecision computation of period matrices and the Abel-Jacobi map of
  superelliptic curves given by $m > 1$ and squarefree $f \in \C[x]$.
 
  Integration along a spanning tree and the special geometry of such curves
  make it possible to compute these objects too high precision performing only
  a few numerical integrations. The resulting algorithm has an excellent
  scaling with the genus and works for several thousand digits of precision.

  \subsection{Reduced small period matrix}

   For a given curve our algorithm computes a small period matrix
   $\tau$ in the Siegel upper half-space $\mathcal{H}_g$ which is arbitrary
   in the sense that it depends on the choice of a symplectic basis made
   during the algorithm.
   
   For applications like the computation of theta functions it is useful to
   have a small period matrix in the Siegel fundamental domain $\mathcal{F}_g \subset
   \mathcal{H}_g$ (see \cite[\S 1.3]{PlaneQuarticsCM}).
  
   We did not implement any such reduction.
   The authors of \cite{PlaneQuarticsCM} give a theoretical sketch of
   an algorithm (Algorithm 1.9) that achieves this reduction step, as well as
   two practical versions (Algorithms 1.12 and 1.14) which work in any genus and have been implemented for $g
   \le 3$. It would be interesting to combine this with our implementation.
  
  \subsection{Generalizations}
  
  We remark that there is no theoretical obstruction to generalizing our
  approach to more general curves. In a first step the algorithm could be
  extended to all complex superelliptic curves given by $m > 1$ and $f \in
  \C[x]$, where $f$ can have multiple roots of order at most $m-1$.
  Although several adjustments would have to be made (e.g.\ differentials,
  homology, integration), staying within the superelliptic setting promises
  a fast and rigorous extension of our algorithm. 
  
  \medskip
  
  We also believe that the strategy employed here (numerical integration between
  branch points combined with information about local intersections) could
  be adapted to completely general algebraic curves given by $F \in \C[x,y]$.
  However, serious issues have to be overcome:
  \begin{itemize}
      \item On the numerical side we no longer have a nice $m$-th root function, it may be replaced by
          Newton's method between branch points (analytic continuation has to be performed on all sheets) and
          Puiseux series expansion around them.
      \item On the geometric side we cannot easily define loops, so that given a set
          of ``half'' integrals each connecting two branch points, we need to combine them in order
          to obtain all at once true loops and a symplectic basis.
          An appropriate notion of shifting number and local intersection is needed here,
          as well as a combination technique.
  \end{itemize}
  We did not investigate further: at this point the advantages
  of superelliptic curves which are utilized by our approach are already lost
 (simple geometry of branch points and $m-1$ integrals at the cost of one), 
  so it is
  not clear whether this approach might be more efficient than other methods.

\newpage
\bibliographystyle{plain}
\bibliography{./main}

\begin{thebibliography}{10}

\bibitem{AbramowitzStegun}
Milton Abramowitz and Irene~A. Stegun.
\newblock {\em Handbook of mathematical functions with formulas, graphs, and
  mathematical tables}, volume~55 of {\em National Bureau of Standards Applied
  Mathematics Series}.
\newblock For sale by the Superintendent of Documents, U.S. Government Printing
  Office, Washington, D.C., 1964.

\bibitem{Magma}
Wieb Bosma, John Cannon, and Catherine Playoust.
\newblock The {M}agma algebra system. {I}. {T}he user language.
\newblock {\em J. Symbolic Comput.}, 24(3-4):235--265, 1997.
\newblock Computational algebra and number theory (London, 1993).

\bibitem{BostMestre88}
Jean-Beno{\^\i}t Bost and Jean-Fran\c{c}ois Mestre.
\newblock Moyenne arithm\'etico-g\'eom\'etrique et p\'eriodes des courbes de
  genre {$1$} et {$2$}.
\newblock {\em Gaz. Math.}, 1(38):36--64, 1988.

\bibitem{BrentZimmermann}
Richard~P. Brent and Paul Zimmermann.
\newblock {\em Modern computer arithmetic}, volume~18 of {\em Cambridge
  Monographs on Applied and Computational Mathematics}.
\newblock Cambridge University Press, Cambridge, 2011.

\bibitem{ChawlaJain68}
M.~M. Chawla and M.~K. Jain.
\newblock Error estimates for {G}auss quadrature formulas for analytic
  functions.
\newblock {\em Math. Comp.}, 22:82--90, 1968.

\bibitem{CMSVEndos}
Edgar Costa, Nicolas Mascot, Jeroen Sijsling, and John Voight.
\newblock Rigorous computation of the endomorphism ring of a jacobian.
\newblock {\em arXiv preprint arXiv:1705.09248}, 2017.

\bibitem{CremonaAGM13}
John~E. Cremona and Thotsaphon Thongjunthug.
\newblock The complex {AGM}, periods of elliptic curves over {$\Bbb{C}$} and
  complex elliptic logarithms.
\newblock {\em J. Number Theory}, 133(8):2813--2841, 2013.

\bibitem{DeconinckvanHoeij01}
Bernard Deconinck and Mark van Hoeij.
\newblock Computing {R}iemann matrices of algebraic curves.
\newblock {\em Phys. D}, 152/153:28--46, 2001.
\newblock Advances in nonlinear mathematics and science.

\bibitem{FrauendienerKlein2011}
J{\"o}rg Frauendiener and Christian Klein.
\newblock Algebraic curves and riemann surfaces in matlab.
\newblock In {\em Computational approach to Riemann surfaces}, pages 125--162.
  Springer, 2011.

\bibitem{FrauendienerKlein2015}
J{\"o}rg Frauendiener and Christian Klein.
\newblock Computational approach to hyperelliptic riemann surfaces.
\newblock {\em Letters in Mathematical Physics}, 105(3):379--400, 2015.

\bibitem{Johansson2013arb}
F.~Johansson.
\newblock {A}rb: a {C} library for ball arithmetic.
\newblock {\em ACM Communications in Computer Algebra}, 47(4):166--169, 2013.

\bibitem{PlaneQuarticsCM}
Pinar Kilicer, Hugo Labrande, Reynald Lercier, Christophe Ritzenthaler, Jeroen
  Sijsling, and Marco Streng.
\newblock Plane quartics over q with complex multiplication.
\newblock {\em arXiv preprint arXiv:1701.06489}, 2017.

\bibitem{KB2002}
Greg Kuperberg.
\newblock Kasteleyn cokernels.
\newblock {\em Electronic Journal of Combinatorics}, 9, 2002.

\bibitem{Labrande16}
Hugo Labrande.
\newblock {\em {Explicit computation of the Abel-Jacobi map and its inverse}}.
\newblock Theses, {Universit{\'e} de Lorraine ; University of Calgary},
  November 2016.

\bibitem{lmfdb}
The {LMFDB Collaboration}.
\newblock The l-functions and modular forms database.
\newblock \url{http://www.lmfdb.org}, 2013.
\newblock [Online; accessed 16 September 2013].

\bibitem{Mascot13}
Nicolas Mascot.
\newblock Computing modular {G}alois representations.
\newblock {\em Rend. Circ. Mat. Palermo (2)}, 62(3):451--476, 2013.

\bibitem{Miranda1995}
Rick Miranda.
\newblock {\em Algebraic Curves and Riemann Surfaces (Graduate Studies in
  Mathematics, Vol 5)}.
\newblock American Mathematical Society, 4 1995.

\bibitem{Molin2010}
Pascal Molin.
\newblock {\em Intégration numérique et calculs de fonctions L}.
\newblock PhD thesis, Université de Bordeaux I, 2010.

\bibitem{githubhcperiods_2017_833727}
Pascal Molin and Christian Neurohr.
\newblock hcperiods: arb and magma packages for periods of superelliptic
  curves.
\newblock \url{https://doi.org/10.5281/zenodo.833727}, July 2017.

\bibitem{CT1996}
Christopher Towse.
\newblock Weierstrass points on cyclic covers of the projective line.
\newblock {\em Transactions of the American Mathematical Society},
  348(8):3355--3378, 1996.

\bibitem{vanWam1998}
Paul Van~Wamelen.
\newblock Equations for the jacobian of a hyperelliptic curve.
\newblock {\em Transactions of the American Mathematical Society},
  350(8):3083--3106, 1998.

\bibitem{vanWamelen06}
Paul~B. van Wamelen.
\newblock Computing with the analytic {J}acobian of a genus 2 curve.
\newblock In {\em Discovering mathematics with {M}agma}, volume~19 of {\em
  Algorithms Comput. Math.}, pages 117--135. Springer, Berlin, 2006.

\end{thebibliography}

\end{document}